\documentclass[10pt,a4paper]{article}
\usepackage[a4paper,left=2.5cm,right=2.5cm,top=2.5cm,bottom=2.5cm]{geometry}
\usepackage[british]{babel}
\usepackage[utf8]{inputenc}
\usepackage[T1]{fontenc}
\usepackage{csquotes} 
\usepackage[backend=biber, style=alphabetic, maxbibnames=4,minbibnames=4, sorting=nyt]{biblatex}
\usepackage{amsfonts}
\usepackage{amssymb}
\usepackage{amsmath}
\usepackage{tikz-cd}
\usepackage{mathtools,bm,bbm}
\usepackage[mathscr]{eucal}
\usepackage{xfrac,dsfont}
\usepackage{stackengine,units}
\usepackage{relsize}
\DeclareMathAlphabet{\mathbbmsl}{U}{bbm}{m}{sl}
\usepackage[colorlinks=true, allcolors=blue]{hyperref}
\usepackage{enumitem}
\usepackage{calc}

\usepackage{amsthm}
\newtheorem{theorem}{Theorem}[section]
\newtheorem{lemma}[theorem]{Lemma}
\newtheorem{proposition}[theorem]{Proposition}
\newtheorem{definition}[theorem]{Definition}
\newtheorem{example}[theorem]{Example}
\newtheorem{corollary}[theorem]{Corollary}
\newtheorem{remark}[theorem]{Remark}
\newtheorem{notation}[theorem]{Notation}
\newtheorem*{theorem*}{Theorem}

\newcommand{\N}{\mathbb{N}}
\newcommand{\Z}{\mathbb{Z}}
\newcommand{\R}{\mathbb{R}}
\newcommand{\C}{\mathbb{C}}

\newcommand{\frg}{\mathfrak{g}}

\newcommand{\frt}{\mathfrak{t}}
\newcommand{\frk}{\mathfrak{k}}
\newcommand{\fra}{\mathfrak{a}}
\newcommand{\frp}{\mathfrak{p}}

\newcommand{\tenr}{\otimes_\mathbb{R}}
\newcommand{\tenc}{\otimes_\mathbb{C}}
\newcommand{\dolb}{\bar{\partial}}

\DeclareMathOperator{\Tr}{Tr}
\DeclareMathOperator{\image}{Im}
\DeclareMathOperator{\kernel}{Ker}
\DeclareMathOperator{\GL}{GL}
\DeclareMathOperator{\Sp}{Sp}
\DeclareMathOperator{\Sym}{Sym}

\DeclareMathOperator{\SU}{SU}
\DeclareMathOperator{\End}{End}
\DeclareMathOperator{\SO}{SO}
\DeclareMathOperator{\Id}{Id}
\DeclareMathOperator{\Ad}{Ad}
\DeclareMathOperator{\Aut}{Aut}
\DeclareMathOperator{\proj}{proj}
\DeclareMathOperator{\Hom}{Hom}

\newcommand{\glxsp}[1]{\GL({#1},\mathbb{H})\times\Sp (1)}
\newcommand{\glsp}[1]{\GL({#1},\mathbb{H})\cdot\Sp (1)}
\newcommand{\spsp}[1]{\spn{#1}\cdot\Sp (1)}
\newcommand{\hnnsp}[1]{\mathbb{H}^{{#1}\times {#1}}\cdot\Sp (1)}
\newcommand{\glh}[1]{\GL({#1},\mathbb{H})}
\newcommand{\glc}[1]{\GL({#1},\mathbb{C})}

\newcommand{\spn}[1]{\Sp({#1})}
\newcommand{\symh}[1]{\Sym^{#1} H}
\newcommand{\exte}[1]{\Lambda^{#1}E^{\ast}}

\newcommand{\spanvector}[1]{\langle #1 \rangle}

\newcommand{\textover}[3][l]{%
 \makebox[\widthof{#3}][#1]{#2}
}
\begin{document}
\title{Fixed Point Theorems in Quaternionic Geometry}
\author{Niklas Henningsen\thanks{Mathematisch-Naturwissenschaftliche Fakult\"at, Mathematisches Institut, Heinrich-Heine-Universit\"at D\"usseldorf, Universit\"atsstraße
1, 40225 D\"usseldorf, Germany\\ Email: \texttt{niklas.henningsen@hhu.de}}}
\date{15 September 2026}
\maketitle

\begin{abstract}        
  \noindent We consider an application of the Atiyah-Bott fixed point theorem in quaternionic geometry.\\
  We review the relation of the general Salamon complex of a quaternionic manifold and the Dolbeault complex with coefficients in a certain holomorphic vector bundle of its twistor space. Furthermore, we investigate which maps are suitable for the application of the fixed point theorem. As a result, we obtain a fixed point formula for both complexes and apply it to Wolf spaces.
\end{abstract}


\section{Introduction}
\label{quatsect1}
Just as every complex manifold comes equipped with the Dolbeault complex, every quaternionic manifold also carries an elliptic complex of differential operators induced by its quaternionic structure. This complex is called the Salamon complex and its differential operator $\delta$ plays a role analogous to the Dolbeault operator $\dolb$: In the same way that the Newlander-Nirenberg theorem characterises integrable almost complex structures by the condition $\dolb^2=0$, a result of Salamon \cite[Theorem~2.2]{salamonQM} shows that an almost quaternionic manifold is quaternionic if and only if $\delta^2=0$.\\
In \cite{salamonPhD}, Salamon defined a more general elliptic complex, called the general Salamon complex, associated with the choice of an $r\in 2\cdot\N_0$ and a complex vector bundle equipped with a linear connection whose curvature only has form components in a certain subbundle.\\
In \cite{kkquattorsion}, K{\"o}hler and Weingart defined the quaternionic torsion of the general Salamon complex. They also computed the torsion for certain quaternionic manifolds, the so-called Wolf spaces, which are symmetric spaces that were characterised by Wolf in \cite{wolf}. These spaces were also studied in \cite{AlekseevskyQuatSpaces} and \cite{AlekseevskyCortesQuat}. In \cite{AlekseevskyQuatSpaces}, it was shown that every compact homogeneous quaternionic K\"ahler manifold with non-zero Ricci curvature is automatically a Wolf space. In \cite{AlekseevskyCortesQuat}, Alekseevsky and Cort{\'e}s proved that every quaternionic K\"ahler manifold $M$ admitting a transitive action of an unimodular group $G$ is either flat, i.e.\ the Riemannian product of a torus and a Euclidean space, or it is a quaternionic K\"ahler symmetric space $G/H$. In a recently submitted article \cite{brendlesemmelmann2025quaternionic}, Brendle and Semmelmann showed that any compact quaternionic K\"ahler manifold with positive scalar curvature and non-negative sectional curvature is isometric to a symmetric space.\\

As described in \cite{salamonPhD} and \cite{salamonQM}, the general Salamon complex is closely related to the Dolbeault complex, as it can be embedded into the Dolbeault complex with coefficients in a certain holomorphic vector bundle on a complex manifold associated to the quaternionic manifold, the so-called Twistor space. With respect to this embedding, the differential operator $\delta_{F,r}$ of the general Salamon complex is induced by the Dolbeault operator and the embedding gives an isomorphism of their cohomologies \cite{salamonPhD}.\\
In \cite{SpacilQuatComplexes}, Sp{\'a}{\v{c}}il obtained a formula for the (non-equivariant) index of the Salamon complex in terms of characteristic clases. In \cite{BastonQuatComplexes}, further quaternionic complexes are defined and the (non-equivariant) index for Hyperk\"ahler manifolds is also calculated in terms of characteristic classes. Atiyah and Bott established a fixed point theorem for elliptic complexes in \cite{atiyahbott1} and subsequently applied it to the Dolbeault complex with coefficients in a holomorphic vector bundle \cite{atiyahbott2}. This motivates us to apply the Atiyah-Bott fixed point theorem, which gives a formula for the equivariant index, to the general Salamon complex of a compact quaternionic manifold.\\
For the application of the fixed point theorem in quaternionic geometry, it is essential to determine which maps are admissible. In this regard, \cite{twistorialmaps} gave a definition of quaternionic maps and studied them and their relation to holomorphic maps between twistor spaces. However, in the case $r=0$, we can even use a slightly more general definition of quaternionic maps. In \cite{ChenLiQuaternionicMaps}, a slightly different definition of quaternionic maps for Hyperk\"ahler manifolds, which are a special type of quaternionic manifolds, was given. However, the definition given in \cite{ChenLiQuaternionicMaps} can only be applied to Hyperk\"ahler manifolds and differs significantly from the definition we use, as it is based on a globally uniform transformation of the hypercomplex structure at each point of the manifold.\\

In Section~\ref{quatsect3} of this part, we will deal with the construction of the general Salamon complex. In this process, we give proofs for the results of \cite{salamonPhD} for the general case, i.e.\ including the choice of $r\in 2\cdot\N_0$ and a complex coefficient vector bundle as mentioned above, as these results were only proven for the special case of $r=0$ and a trivial complex line bundle as the coefficient bundle in \cite{salamonPhD}. We also correct a missing factor in the definition of a homomorphism used in the definition of $\delta_{F,r}$ in \cite{salamonPhD}.\\
The differential operator $\delta_{F,r}$ (see Proposition~\ref{quatlabel25}) is an example of an invariant differential operator in quaternionic geometry because it does not depend on the choice of the torsion-free connection. Therefore, we must also re-examine Salamon's choice of the power of the one-dimensional representation that makes the differential operator invariant, as this will play an important role later on and to avoid possible confusion when comparing with the choice made in \cite{kkquattorsion}.\\

Motivated by Atiyah and Bott's application of their fixed point theorem in complex geometry \cite{atiyahbott2}, we apply the Atiyah-Bott fixed point theorem to the general Salamon complex.
\begin{itemize}
    \item The torsion-free $\glsp{n}$-structure of a quaternionic manifold $M$ will be denoted by $P\subset \GL (M)$ and its twistor space is a certain subbundle $Z\subset \End(TM)$ with $\pi_Z\colon Z\rightarrow M$.
    \item A complex vector bundle $F\rightarrow M$ together with a connection $\nabla^F$ is called a $B_2$-vector bundle if the curvature of $\nabla^F$ only has form components in $B_2$, which is defined by $B_{2,x}:=\bigcap_{J_z\in Z_x}\Lambda^{1,1}T^{\ast}_xM_{J_z}\subset \Lambda^2 T^\ast_x M\otimes_\R \C$ for every $x\in M$ and complex structures $J_z\in Z_x:=\pi_Z^{-1} (x)$, where $\Lambda^{1,1}T^{\ast}_xM_{J_z}$ are the $(1,1)$-forms with respect to $J_z$ on $T_xM$.
    \item For every $x\in M$, the derivative $T_x\gamma$ of a quaternionic map $\gamma \colon M\rightarrow M$ is a quaternionic linear map of the vector spaces $T_xM$ and $T_{\gamma (x)}M$ (See Definition~\ref{quatlabel26} and Remark~\ref{quatlabel28}).
    \item For $r\in 2\cdot\N_0$, let $V^{k,r}$ be the associated vector bundle to the representation $\Lambda^k E^\ast \otimes_\C \Sym^{k+r} H\otimes_\C\phi^r$ of $\glsp{n}$, which is defined in Section~\ref{quatsect2.2}. The differential operator of the general Salamon complex is a map $\delta_{F,r}\colon \Gamma(M,V^{k,r}\tenc F)\rightarrow \Gamma(M,V^{k+1,r}\tenc F)$ defined in Proposition~\ref{quatlabel25}.
    \item For vector bundle $V\rightarrow M$ and a $C^\infty$-map $\gamma \colon M\rightarrow M$, let $\gamma^V\colon\gamma^\ast V\rightarrow V$ be a vector bundle homomorphism. We define $\tau^V:=\gamma^V\circ \gamma^\ast$, where $\gamma^\ast$ is the pullback of sections.
\end{itemize}
We obtain the following theorem that we will prove in Section~\ref{quatsect4}.
\begin{theorem}[Fixed point theorem for the Salamon complex]
\label{quatlabel33}
    Let $r\in 2\cdot\N_0$, $M$ be a compact quaternionic manifold and $\pi_F\colon F\rightarrow M$ a $B_2$-vector bundle. Furthermore, let $\gamma\colon M\rightarrow M$ be a quaternionic map such that $\det\nolimits_\R (1-T_x\gamma)\neq 0$ for all fixed points $x\in M^\gamma$ and, for $r\neq 0$, $\det\nolimits_\R (T_x\gamma)\neq 0$ for all $x\in M$. With the maps $\gamma^{V^{\bullet,r}},\gamma^F$ and $\tau^{V^{\bullet,r}},\tau^F$ as in Section~\ref{quatsect4.1}, such that $\tau^{V^{\bullet,r}\tenc F} := \tau^{V^{\bullet,r}}\tenc\tau^F$ commutes with $\delta_{F,r}$, we obtain
    \begin{align*}
        \Tr_\mathrm{s} \left(\tau^{V^{\bullet,r}\tenc F}_{\vert H^\bullet_{\delta_{F,r}}(M)}\right) = \sum\limits_{x\in M^\gamma} \omega (x),
    \end{align*}
    where $T_x\gamma = p_x\circ (v\mapsto T_xve^{-it_x}) \circ p_x^{-1}$ for $x\in M^\gamma$, $p_x\in P_x$ as in Corollary~\hyperref[quatlabel17]{\ref*{quatlabel17}.(\ref*{quatlabel18})} and
    \begin{align*}
        \omega (x) := \Tr (\gamma^F_{\vert x})\cdot \left(\frac{\det_\C (T_x)^{-\frac{r}{2(n+1)}}e^{irt_x}}{(1-e^{-2it_x})\det_\C (1-e^{-it_x}T_x)} + \frac{\det_\C (T_x)^{-\frac{r}{2(n+1)}}e^{-irt_x}}{(1-e^{2it_x})\det_\C (1-e^{it_x}T_x)}\right)
    \end{align*}
    for $e^{it_x}\not\in \lbrace -1,1\rbrace$ and
    \begin{align*}
        \omega(x) := \Tr (\gamma^F_{\vert x})\cdot\frac{\det_\C (T_x)^{-\frac{r}{2(n+1)}}(r+1-\Tr ((1-T_x)^{-1}T_x))}{\det_\C(1-T_x)}
    \end{align*}
    for $e^{it_x}\in \lbrace -1,1\rbrace$.
\end{theorem}
Due to the relation of the general Salamon complex and the Dolbeault complex of its Twistor space, we obtain a fixed point theorem for the twistor space (Theorem~\ref{quatlabel23}), which is the classical fixed point theorem by Atiyah and Bott.\\
In Section~\ref{quatsect4.4}, we consider the case of Wolf spaces as an application of our fixed point theorem in quaternionic geometry. This application gives a sum over formal characters.\\

This work is part of the author’s PhD thesis that is currently under review.\\

\textbf{Acknowledgements:} I would like to thank Prof.\ Dr.\ Kai K\"ohler for the opportunity to conduct this project as part of my doctoral studies and for his support during my time at Heinrich Heine University Düsseldorf.
\section{Quaternionic Vector Spaces}
\label{quatsect2}
In the first sections, we collect some general information and results from \cite{twistorialmaps} and \cite{salamonPhD} that we will need later on and which will also serve as motivation for later definitions.\\

The quaternions $\mathbb{H}:=\spanvector{1,i,j,k}_\R$ form an $\R$-algebra, where $i^2=-1=j^2$, $k:=ij=-ji$. Now $\SO (3)$ acts canonically on $Q_\mathbb{H}:=\spanvector{i,j,k}_\R$ and trivially on $\R\subset\mathbb{H}$. Furthermore, $\mathbb{H}$ carries a canonical scalar product, whose unit sphere we denote by $\spn{1}$, for which $\nu\in Q_\mathbb{H}$ satisfies $\nu^2=-1$ if and only if $\nu\in\spn{1}\cap Q_\mathbb{H}$. We define $Z_\mathbb{H}:=Q_\mathbb{H}\cap \spn{1}$.
\begin{definition}
    Let $V$ be a $4n$-dimensional real vector space. A subvector space $Q_V\subset\End_\R(V)$ is called a \textbf{quaternionic structure} if there exists an algebra isomorphism $Q_V\oplus \R\cdot\Id_V\cong \mathbb{H}$. A \textbf{hypercomplex structure} is a triple $(I,J,K)\in Q_V^3$ such that $I^2=-1=J^2, K=IJ=-JI$. We denote by $Z_V$ all $I\in Q_V$ with $I^2=-1$.
\end{definition}
\noindent On a quaternionic structure $Q_V$, $\SO (3)$ acts naturally, and for two hypercomplex structures $(I,J,K)$ and $(I',J',K')$, there exists exactly one element of $A\in \SO (3)$ such that $(I',J',K') = (A(I),A(J),A(K))$. This provides a transitive action of $\SO (3)$ on the set of hypercomplex structures.
\begin{example}
    For $\lambda\in\mathbb{H}$, we write $R_\lambda(v):=v\lambda$ for the componentwise multiplication with $\lambda$ from the right for $v\in\mathbb{H}^n$. The canonical quaternionic structure on $\mathbb{H}^n$ is $Q_{\mathbb{H}^n}=\spanvector{R_{-i},R_{-j},R_{-k}}$. There is also a canonical scalar product on $\mathbb{H}^n$.
\end{example}
\subsection{Quaternionic Linear Maps}
\label{quatsect2.1}
For a vector space with a quaternionic structure, there is no canonical hypercomplex structure, so the definition of quaternionic maps is somewhat different from the definition of complex-linear maps. We want a definition that is independent of the choice of a specific hypercomplex structure.
\begin{definition}
    \label{quatlabel27}
    Let $V_1,V_2$ be two vector spaces with quaternionic structure. We call a $\R$-linear map $\gamma\colon V_1\rightarrow V_2$ \textbf{quaternionic} if hypercomplex structures $(I,J,K)\in Z_{V_1}^3$, $(I',J',K')\in Z_{V_2}^3$ exist with $\gamma\circ I= I'\circ \gamma,\gamma\circ J= J'\circ \gamma, \gamma\circ K= K'\circ \gamma$.
\end{definition}
\begin{proposition}[{\cite[Proposition~1.6]{twistorialmaps}}]
    For every $4n$-dimensional real vector space with quaternionic structure, there is an isomorphism $V\xrightarrow{\cong} \mathbb{H}^n$, which is a quaternionic map.
\end{proposition}
\noindent Every quaternionic matrix $T\in\mathbb{H}^{n\times n}$ gives a quaternionic map $\gamma(v):=Tv$ for $v\in\mathbb{H}^n$. Let $\glh{n}$ denote the invertible quaternionic matrices and let $\spn{n}\subset\glh{n}$ denote the matrices that leave the scalar product on $\mathbb{H}^n$ invariant.\\
Now we want to determine all quaternionic linear maps $\gamma\colon \mathbb{H}^n\rightarrow \mathbb{H}^n$. The following lemma is a summary of \cite[Section~1]{twistorialmaps} and uses the fact that $\spn{1}\rightarrow \SO (3)$, $\lambda\mapsto (\nu\mapsto \lambda\nu\lambda^{-1})$ is a $2$-fold covering according to \cite[Chapter~5]{curtis}.
\begin{lemma}
\label{quatlabel1}
    For a $\R$-linear map $\gamma \colon \mathbb{H}^n\rightarrow \mathbb{H}^n$, the following statements are equivalent:
    \begin{enumerate}
        \item $\gamma$ is a quaternionic map.
        \item There exist $T\in\mathbb{H}^{n\times n}$ and $\lambda\in \spn{n}$ such that $\gamma (v)=Tv\lambda^{-1}$ for all $v\in \mathbb{H}^n$.
        \item There is a map $\varphi \colon Z_\mathbb{H}\rightarrow Z_\mathbb{H}$ with $\varphi (\nu )\circ \gamma = \gamma \circ \nu$ for all $\nu\in Z_\mathbb{H}$.
    \end{enumerate}
    In particular, $\varphi(\nu)=\lambda\nu\lambda^{-1}$ and $\varphi\in\SO (3)$ if $\gamma\neq 0$.
\end{lemma}
Let 
\begin{align*}
    \mathbb{H}^{n\times n}\cdot\spn{1} :=& \lbrace v\mapsto Tv\lambda^{-1}\mid T\in\mathbb{H}^{n\times n},\lambda\in\spn{1}\rbrace\subset\mathrm {End}_\R(\mathbb{H}^n),\\
    \glsp{n} :=& \lbrace v\mapsto Tv\lambda^{-1}\mid T\in\glh{n},\lambda\in\spn{1}\rbrace. 
\end{align*}
Note that $\glsp{n}=(\glxsp{n})/\lbrace \pm (\Id_{\mathbb{H}^n},1)\rbrace$ holds. We denote an element of $\mathbb{H}^{n\times n}\cdot\spn{1}$ by $[T,\lambda]$.
\subsection{Representations}
\label{quatsect2.2}
Before defining the representations which will be used frequently in the following, we first consider the embedding of quaternionic matrices into complex matrices, as done in \cite{quatmatrices}.\\
With respect to the isomorphism $\mathbb{H}^n\rightarrow \C^{2n},u+jv\mapsto\begin{pmatrix}
    u\\v
\end{pmatrix}$, the operation of a quaternionic matrix $A+jB$ with $A,B\in \C^{n\times n}$ on $\mathbb{H}^n$ corresponds exactly to the operation of the complex matrix $\begin{pmatrix}
    A&-\overline{B}\\ B&\overline{A}
\end{pmatrix}$ on $\C^{2n}$. Let $\mathbb{H}^{n\times n}\hookrightarrow\C^{2n\times 2n}$, $A+jB\mapsto \begin{pmatrix}
    A&-\overline{B}\\ B&\overline{A}
\end{pmatrix}$. We can consider $\mathbb{H}^{n\times n}$ and its subspaces $\glh{n},\spn{n}$ as subspaces of $\C^{2n\times 2n}$.\\
With the help of $J_r\colon \C^{2n}\rightarrow\C^{2n}$, $\begin{pmatrix}
        u\\v
    \end{pmatrix}\mapsto\begin{pmatrix}
        -\overline{v}\\ \overline{u}
    \end{pmatrix}$, we have
    \begin{align*}
        \mathbb{H}^{n\times n} &= \lbrace T\in \C^{2n\times 2n}\mid T\circ J_r=J_r\circ T\rbrace,\\
        \glh{n} &= \mathbb{H}^{n\times n}\cap\glc{2n},\\
        \spn{n} &= \mathbb{H}^{n\times n}\cap \operatorname{U} (2n).
    \end{align*}
\noindent The map $J_r$ corresponds to the scalar multiplication of a vector with $j\in \mathbb{H}$ from the right.
\begin{remark}
    One can show that $T$ and $\overline{T}$ are similar for any $T\in\mathbb{H}^{n\times n}$ with the help of the matrix $\begin{pmatrix} 0&-\Id_{\C^n} \\ \Id_{\C^n}&0\end{pmatrix}$. According to \cite[Proposition~4.2]{quatmatrices}, we even have $\det\nolimits_\C (T)\geq 0$. In particular, the eigenvalues of any $\lambda \in \spn{1}$ are complex conjugates of each other, thus the eigenvalues are $e^{it}$ and $e^{-it}$ for some $t \in \R$.
\end{remark}
This allows use to define the following representations:
\begin{itemize}
    \item $H:=\C^2$ denotes the standard representation of $\spn{1}$,
    \item $E:=\C^{2n}$ denotes the standard representation of $\glh{n}$,
    \item $\phi^r:=\C$ denotes representation $\glh{n}\rightarrow \glc{1}=\C\setminus \lbrace 0\rbrace$, $T\mapsto (\det\nolimits_\C (T))^{\frac{r}{2(n+1)}}$ for $r\in \R$.
\end{itemize}
Their dual representations are denoted by $H^\ast,E^\ast,\phi^{-r}=(\phi^r)^\ast$. 
\begin{notation}
    Let $k\in \N_0$ and $V$ be a vector space. We denote the symmetric tensors by $\Sym^k V$. For $v_1,\ldots,v_k\in V$, we define $v_1\cdot\ldots\cdot v_k:=\sum\limits_{\nu\in S_k} v_{\nu (1)}\otimes \ldots \otimes v_{\nu (k)}$ with the symmetric group $S_k$ of permutations of $\lbrace 1,\ldots,k\rbrace$. Thus, for $v_1,v_2\in V$, we define $v_1^{l}\cdot v_2^{k-l}:=\underbrace{v_1\cdot \ldots \cdot v_1}_{l\text{-times}}\cdot\underbrace{v_2\cdot \ldots \cdot v_2}_{(k-l)\text{-times}}$. Note that $v_1^k=k!\underbrace{v_1\otimes \ldots\otimes v_1}_{k\text{-times}}$ holds. We denote the $k$-th exterior power of $V^\ast$ by $\Lambda^k V^\ast$, and we have $\alpha_1\wedge\ldots\wedge \alpha_k = \sum\limits_{\nu \in S_k} \operatorname{sgn} (\nu) \alpha_{\nu (1)}\otimes\ldots\otimes \alpha_{\nu (k)}$ for $\alpha_1,\ldots,\alpha_k\in V^\ast$.\\
    We consider these spaces as a subspace of $V^{\otimes k}$ and $(V^\ast)^{\otimes k}$, respectively.\\
    For $T\in\End(V)$, we write $T^{\otimes k}$ for the map $v_1\otimes \ldots\otimes v_k\mapsto (Tv_1)\otimes \ldots\otimes (Tv_k)$ on $V^{\otimes k}$ and its subspaces that are invariant under this map.
\end{notation}
All representations mentioned above are irreducible, and the same holds for the tensor product $\Lambda^{s}E^\ast \otimes_\mathbb{C} \Sym^{t}H \otimes_\mathbb{C} \phi^{r}$ for $s,t\in\N_0,r\in\R$, which is a representation of $\glxsp{n}$. For $s=k$, $t=k+r$ with $k\in\N_0,r\in 2\cdot\N_0$, it is even a representation of $\glsp{n}$. For example, for $k=1$ and $r=0$ we obtain the representation $E\tenc H$ of $\glsp{n}$.\\
Now $H$ has a $\spn{1}$-invariant symplectic form $\omega_{H^\ast}(h,h'):=-h^Tjh'$ with $j:=\begin{pmatrix}
    0&-1\\1&0
\end{pmatrix}$ for $h,h'\in H$, as described in \cite{kkquattorsion}.
\begin{definition}
    We call non-zero vectors $h_1,h_2\in H$ a \textbf{standard basis} of $H$ if $h_2=J_r(h_1)$ and $\omega_{H^\ast}(h_1,h_2)=1$.
\end{definition}
\begin{remark}
\label{quatlabel29}
    The map $h\mapsto\omega_{H^\ast} (h,\cdot)$ is a $\spn{1}$-invariant isomorphism of the representations $H$ and $H^\ast$. Because of $H\cong H^\ast$ we also have $\Lambda^2H^\ast\cong \Lambda^2H$ and for $\omega_{H^\ast}$ there is a corresponding $\omega_H\in\Lambda^2H$. For a standard basis $h_1,h_2$, we have $\omega_H=h_1\wedge h_2$ and $\Lambda^2 H$ is a trivial representation using the isomorphism $\Lambda^2H\rightarrow \C$, $\omega_H\mapsto 1$.
\end{remark}
\begin{remark}
\label{quatlabel30}
    With the action as in Lemma~\ref{quatlabel1}, $\mathbb{H}^n$ is a representation of $\glxsp{n}$ and $\glsp{n}$, just like $\R^{4n}$, since we have $\mathbb{H}^n\cong\C^{2n}\cong \R^{4n}$.\\    
    It can be shown that the representations $E\tenc H$ and $\R^{4n}\tenr\C$ are isomorphic by considering the eigenspaces of $J_r\otimes J_r$ on $E\tenc H$. This helps to decompose $\Lambda^{\bullet} (\R^{4n}\tenr \C)^\ast\cong \Lambda^{\bullet}(E^\ast \tenc H)$ into irreducible representations because we can consider $\glh{n}$ and $\spn{1}$ separately, see also \cite[p.~12]{salamonPhD} and \cite{kkquattorsion}. An important summand in this decomposition is $\exte{k}\tenc \symh{k}$. This helps later on to decompose complexified differential forms.\\
    This isomorphism of representations not only commutes with the action of $\glsp{n}$, but also with maps of $\hnnsp{n}$. With respect to $E\tenc H\cong \R^{4n}\tenr \C$, the eigenspaces of a complex structure $R_\nu\in Z_{\mathbb{H}^n}\subset \End_\R(\mathbb{H}^n)\cong\End_\R (\R^{4n})$ are given by $E\tenc \C h_{\pm i}$ for a non-zero eigenvector $h_{\pm i}$ corresponding to the eigenvalue $\pm i$ of $\nu\in Z_\mathbb{H}\subset \spn{1}$.
\end{remark}
\begin{remark}
    The complex conjugation on $R^{4n}\otimes_\R\C$, i.e.\ $v\otimes_\R z\mapsto v\otimes_\R\overline{z}$, corresponds to the map $J_r\otimes_\C J_r$ under the isomorphism above.
\end{remark}
\section{Quaternionic Manifolds} \label{quatsect3}
\subsection{$G$-Structures and Connections}\label{quatsect3.1}
Before we move on to the definition of quaternionic manifolds, let us collect some relevant definitions and results from \cite{baumeichfeld} in the next sections that are important for further understanding.\\
Let $M$ be a smooth manifold of dimension $m$. Then there exists a principal bundle $\GL(M) \to M$ with structure group $\GL(m,\mathbb{R})$, known as the \textbf{frame bundle} of $M$. The fibre of $\GL(M)$ over a point $x \in M$ consists of all linear isomorphisms $\mathbb{R}^m \to T_xM$. The existence of subbundles $P \subset \GL(M)$, which are reductions of the frame bundle to a Lie subgroup $G \subset \GL(m,\mathbb{R})$, often induces additional geometric structures on the manifold $M$. Such a subbundle $P \subset \GL(M)$ with structure group $G$ is called a \textbf{$G$-structure}.\\
A classical example is an $\operatorname{O}(m)$-structure, which corresponds to the choice of a Riemannian metric on $TM$, and which always exists.\\

Let $G$ be a Lie group and let $\pi_P\colon P \to M$ be a principal $G$-bundle over $M$. The map
\begin{align*}
    P\times \mathfrak{g}&\rightarrow T^VP:=\kernel (T\pi_P)\\
    (p,X)&\mapsto \frac{\partial}{\partial t} (p\cdot \exp(tX))_{\vert t=0}
\end{align*}
is a vector bundle isomorphism. For $X \in \mathfrak{g}$, we denote by $\widetilde{X}\colon P\rightarrow T^VP\subset TP$ the vector field on $P$ induced in this isomorphism.
\begin{definition}
\begin{enumerate}
    \item A \textbf{connection} on a principal bundle $\pi_P\colon P \rightarrow M$ is a choice of a horizontal tangent bundle $T^H P \subset TP$, i.e.\ a subbundle of $TP$ such that $T_eR_g(T_p^H P) = T_{p \cdot g}^H P$ and $TP = T^V P \oplus T^H P$.
    \item A \textbf{connection form} on the principal bundle $\pi_P\colon P \rightarrow M$ is a 1-form $\omega \in \mathfrak{A}^1(P, \mathfrak{g})$, where $\mathfrak{g}$ is the Lie algebra of $G$, such that:
    \begin{itemize}
        \item $R_g^* \omega = \Ad(g^{-1}) \circ \omega$ for all $g \in G$,
        \item $\omega(\widetilde{X}) = X$ for all $X \in \mathfrak{g}$.
    \end{itemize}
    We denote by $\mathcal{C}(P)$ the space of all connection forms on $P$.
\end{enumerate}
\end{definition}
\begin{lemma}
    There is a one-to-one correspondence of connection forms and connections:
    \begin{enumerate}
        \item If $T^H P$ is a connection on $P$, then $\omega_p(\widetilde{X}_p \oplus Y_p) := X$ for $X \in \mathfrak{g}$ and $Y_p \in T_p^H P$ defines a connection form.
        \item If $\omega \in \mathcal{C}(P)$ is a connection form, then $T^H P := \kernel(\omega)$ defines a connection on $P$.\\
        Note: $\omega \in \mathfrak{A}^1(P, \mathfrak{g})$, i.e.\ $\omega_p \in \End_\mathbb{R}(T_pP, \mathfrak{g})$.
    \end{enumerate}
\end{lemma}
\noindent For this reason, we will simply typically refer to a connection on $P$ from now on.\\
For any manifold $F$ on which $G$ acts from the left, there exists a fibre bundle $P \times_G F := (P \times F)/G \rightarrow M \cong P/G$, called the \textbf{associated fibre bundle} to $F$. We have $[p,f]=[p\cdot g,g^{-1}\cdot f]$ for all $g\in G$. We use the shorthand notation $\underline{F}:=P \times_G F$ when a $G$-structure is fixed. \\
For a (complex) representation $V$ of $G$, the bundle $\underline{V}$ is a (complex) vector bundle. Every $G$-equivariant map $\kappa\colon F_1\rightarrow F_2$ induces a fibre bundle morphism $\underline{\kappa}\colon \underline{F_1} \rightarrow \underline{F_2}$ and the associated bundles of two isomorphic $G$-representations are also isomorphic.\\
In the case of the frame bundle $\GL(M)$, there is a one-to-one correspondence between linear connections on $TM$ and connections on $\GL(M)$. This is compatible with the identification $TM \cong P \times_G \mathbb{R}^m$ for a $G$-structure $P$, where the action of $G\subset\GL (m,\R)$ is the standard representation on $\R^m$. Moreover, any connection on a $G$-structure $P$ induces a connection on $\GL(M)$.
\begin{definition}
    A connection on a $G$-structure is called \textbf{torsion-free} if the corresponding linear connection on $TM$ is torsion-free. We call a $G$-structure $P$ torsion-free if there exists a connection on $P$ which is torsion-free.
\end{definition}
\begin{definition}
    For $n\geq 2$, a $4n$-dimensional manifold $M$ with a torsion-free $\glsp{n}$-structure $P$ is called a \textbf{quaternionic manifold}. A $4$-dimensional manifold is called \textbf{quaternionic manifold} if it is self-dual. We obtain a subspace
    \begin{align*}
        Q:=\lbrace p\circ R_\nu\circ p^{-1}\mid p\in P, \nu\in Q_\mathbb{H}\rbrace\subset \End (TM)
    \end{align*}
    which makes the tangent space $T_xM$ at each point a vector space with quaternionic structure $Q_x$.
\end{definition}
\begin{remark}
    Because $\glsp{1}\cong \R^+\cdot\SO (4)$ defines a conformal structure, the Levi-Civita connection of any compatible metric induces a torsion-free $\glsp{1}$-structure for $n=1$.
\end{remark}
\begin{remark}
    For Lie subgroups of $\glsp{n}$, there are other special quaternionic manifolds.
    \begin{enumerate}
        \item If a $\spsp{n}$-structure exists, then $M$ is called a quaternionic K\"ahler manifold.
        \item If a $\spn{n}$-structure exists, then $M$ is called a Hyperk\"ahler manifold.
    \end{enumerate}
\end{remark}
Since the $\glsp{n}$-representations $\mathbb{R}^{4n} \tenr \mathbb{C}$ and $E \tenc H$ are isomorphic, we have $TM \tenr \mathbb{C} \cong \underline{E \tenc H}$ and $T^*M \cong \underline{E^\ast \tenc H}$.\\

We review at a few examples of quaternionic manifolds from \cite{salamonPhD} and \cite{salamonQM}.
\begin{example}
    \begin{enumerate}
        \item With its canonical torsion-free $\glsp{n}$-structure $\mathbb{H}^n\times \glsp{n}$, $\mathbb{H}^n$ is a quaternionic manifold.
        \item The projective quaternionic space is defined by
        \begin{align*}
            \mathbf{P}^n\mathbb{H} = (\mathbb{H}^{n+1}\setminus\lbrace 0\rbrace)/\mathbb{H}^\ast,
        \end{align*}
        where $\mathbb{H}^\ast=\mathbb{H}\setminus\lbrace 0\rbrace$ acts by multiplication from the right. With the usual trivialisations of the projective spaces, one can define a well-defined torsion-free $\glsp{n}$-structure.\\
        The projective quaternionic space is also a good example of why it makes more sense to define a quaternionic manifold by the existence of a $\glsp{n}$-structure rather than by the existence of a $\GL(n,\mathbb{H})$-structure, since the latter does not exist for $\mathbf{P}^n\mathbb{H}$.
        \item In \cite{wolf}, Wolf showed that, for every compact simple Lie group $G$ without centre, one can construct a symmetric quaternionic manifold $G/K_1$. We shall focus on these spaces in more detail in Section~\ref{quatsect4.4}.
    \end{enumerate} 
\end{example}
\subsection{Invariant Differential Operators} \label{quatsect3.2}
Before we can describe differential operators in quaternionic geometry that are invariant under the choice of torsion-free connection, as done in \cite{salamonPhD} and \cite{salamonDGQM}, we will continue to cite a few more basic definitions and results from \cite{baumeichfeld}.\\
Let $P$ be a principal bundle with structure group $G$ again and $F$ a manifold on which $G$ acts from the left. A connection on $P$ induces a horizontal tangent bundle $T^H \underline{F}$ for the fibre bundle $\pi_{\underline{F}}\colon \underline{F} \rightarrow M$, and in the case of a (complex) $G$-representation $V$, this also induces a linear connection $\nabla\colon \Gamma(M, \underline{V}) \rightarrow \Gamma(M, T^*M \tenr \underline{V})$ or $\nabla\colon \Gamma(M, \underline{V}) \rightarrow \Gamma(M, (T^*M \tenr \mathbb{C}) \tenc \underline{V})$, respectively. The linear connection on $\underline{V}$ induced by a connection form $\omega \in \mathcal{C}(P)$ is denoted by $\nabla^\omega$ or $\nabla^{\omega,V}$ to distinguish it.\\
We frequently use the identification $\underline{F_1 \times F_2} \cong \underline{F_1} \times \underline{F_2}$ for two spaces $F_1, F_2$ on which $G$ acts. Similarly, we identify $\underline{V_1 \otimes V_2} \cong \underline{V_1} \otimes \underline{V_2}$ for $G$-representations $V_1, V_2$. The construction of linear connections is compatible with this identification to the extent that $\nabla^{\omega,V_1\otimes V_2}$ on $\underline{V_1\otimes V_2}$ translates to $\nabla^{\omega,V_1}\otimes 1+1\otimes \nabla^{\omega,V_2}$ on $\underline{V_1}\otimes \underline{V_2}$.

\begin{definition}
    Let $\rho\colon G \rightarrow \Aut(V)$ be a representation of $G$. A form $\alpha \in \mathfrak{A}^k(P, V)$ is called
    \begin{enumerate}
        \item \textbf{horizontal} if $\alpha_p(X_1, \ldots, X_k) = 0$ whenever one of the vectors $X_j \in T_pP$ is vertical, i.e.\ $X_j \in T_p^V P$,
        \item \textbf{compatible with the representation $\rho$} if $R_g^* \alpha = \rho(g^{-1}) \circ \alpha$ for all $g \in G$. If there is no risk of confusion, we usually write $g^{-1}$ instead of $\rho(g^{-1})$.
    \end{enumerate}
    The set of horizontal forms that are compatible with the representation is denoted by $\mathfrak{A}_{\operatorname{hor}}^k(P, V)^G$.
\end{definition}

\begin{lemma}
    There exists a canonical isomorphism $\mathfrak{A}^k(M, \underline{V}) \rightarrow \mathfrak{A}_{\operatorname{hor}}^k(P, V)^G$.
\end{lemma}
\begin{definition}
The \textbf{total differential} $D_\omega\colon \mathfrak{A}^k(P, V) \rightarrow \mathfrak{A}^{k+1}(P, V)$ associated to $\omega \in \mathcal{C}(P)$ is defined by
\begin{align*}
    (D_\omega \alpha)_p(X_1, \ldots, X_{k+1}) := (d\alpha)_p(\proj_H(X_1), \ldots, \proj_H(X_{k+1}))
\end{align*}
for $p \in P$, $X_1, \ldots, X_{k+1} \in T_pP$, where $\proj_H\colon TP \rightarrow T^V P \oplus T^H P$ is the projection induced by $\omega$. By restriction, we obtain a map $D_\omega\colon \mathfrak{A}_{\operatorname{hor}}^k(P, V)^G \rightarrow \mathfrak{A}_{\operatorname{hor}}^{k+1}(P, V)^G$, and the isomorphism $\mathfrak{A}^k(M, \underline{V}) \cong \mathfrak{A}_{\operatorname{hor}}^k(P, V)^G$ yields a map $d_\omega\colon \mathfrak{A}^k(M, \underline{V}) \rightarrow \mathfrak{A}^{k+1}(M, \underline{V})$, for which in particular the Leibniz rule holds:
\begin{align*}
    d_\omega(\alpha \wedge \beta) = (d\alpha) \wedge \beta + (-1)^k \alpha \wedge (d_\omega \beta)
\end{align*}
for $\alpha \in \mathfrak{A}^k(M)$ and $\beta \in \mathfrak{A}^\bullet(M, \underline{V})$. We now define $\nabla^{\omega, V} := {d_\omega}_{\vert \Gamma(M, \underline{V})}$ to be the linear connection induced by $\omega \in \mathcal{C}(P)$, as mentioned earlier.
\end{definition}
\begin{remark}
    For local sections $p\colon U \rightarrow P$ and $s = [p, v] \in \Gamma(U, \underline{V})$ with $v \in C^\infty(U, V)$, we have
    \begin{align}
    \label{quatlabel2}
        (\nabla^\omega_X s)_{\vert x} = [p_x, (dv)_{\vert x}(X_x) + T_e\rho(\omega(T_xp(X_x)))v_x],
    \end{align}
    for $x \in U$, $X \in T_xM$. Since every section $p\colon U \rightarrow P$ yields a local trivialisation of $\underline{V}$, every section can locally be written in the form above.
\end{remark}
We now return to the discussion of $G$-structures for a Lie subgroup $G \subset \GL(m, \mathbb{R})$. The following lemma will be relevant later on.
\begin{lemma}
\label{quatlabel31}
    Let $\pi_P\colon P \rightarrow M$ be a $G$-structure, $\omega \in \mathcal{C}(P)$ a connection form and $\rho\colon G \rightarrow \Aut(V)$ a $G$-representation.
    \begin{enumerate}
        \item For every $x \in M$, there exists a local section $p\colon U \rightarrow P$ such that $(p^\ast \omega)_{\vert x} = 0$. For $g \in G$, considered as a constant map $g\colon U \rightarrow G$, we also have $((p \cdot g)^\ast \omega)_{\vert x} = 0$.
        \item For every $x \in M$, and given a basis $v_1, \ldots, v_m \in V$, there exists a section $p\colon U \rightarrow P$ such that, for the local frame $(e_j)_j$ defined by $e_j = [p, v_j]$, we have $(\nabla^\omega e_j)_{\vert x} = 0$ for all $j$.
    \end{enumerate}
\end{lemma}
\begin{proof}
\begin{enumerate}
    \item For a local section $p'\colon U\rightarrow P$ and local coordinates $y=(y_1,\ldots, y_n)$ with $x=(0,\ldots,0)$ on $U\ni x$, we have $p'^\ast\omega = \sum\limits_{k=1}^m \omega_k\tenr dy_k$ for $\omega_k\colon U'\rightarrow \frg\subset \R^{m\times m}$. Then put $g_y:=e^{-y_1\omega_1}\cdots e^{-y_m\omega_m}$. We get $g_x=\Id_{\R^m}$ and $T_x g+(p'^\ast \omega)_{\vert x}=0$.\\     
    Due to \cite[Satz~3.3]{baumeichfeld}, we now get the desired statement for $p:=p'\circ g$.
    \item This statement follows easily from (\ref{quatlabel2}) and the first statement.
\end{enumerate}

\end{proof}
As in \cite{salamonPhD} and \cite{salamonDGQM}, we are interested in differential operators $\underline{\kappa} \circ \nabla^\omega \colon \Gamma(M, \underline{V}) \rightarrow \Gamma(M, \underline{W})$, where $\omega$ is a torsion-free connection form of a $G$-structure and $\kappa\colon (\mathbb{R}^n \tenr \mathbb{C})^\ast \tenc V \rightarrow W$ is a $G$-equivariant homomorphism between complex representations $V$ and $W$ of $G$. Our goal is for $\underline{\kappa} \circ \nabla^\omega$ to be independent of the choice of torsion-free connection. For this reason, we must consider the difference $\omega_1 - \omega_2$ of two torsion-free connection forms $\omega_1, \omega_2$ on a $G$-structure $P$.\\

For a $G$-structure $P$, the \textbf{solder form} $\theta$ is defined by $\theta_p(X) := p^{-1}(T_p\pi_P(X))$ for $p \in P$, $X \in \mathbb{R}^m$. Then $\theta \in \mathfrak{A}^1_{\operatorname{hor}}(P, \mathbb{R}^m)^G$ and the form $\Theta_\omega := D_\omega(\theta) \in \mathfrak{A}^2_{\operatorname{hor}}(P, \mathfrak{g})^G$ is called the \textbf{torsion form} of $\omega \in \mathcal{C}(P)$. The vanishing of the torsion form of $\omega$ corresponds exactly to classical torsion-freeness of the connection on $TM$ induced by $\omega$. We have $\Theta_\omega = d\theta + \omega \wedge \theta$, and for two torsion-free connection forms $\omega_1, \omega_2$, we get $0 = \Theta_{\omega_1} - \Theta_{\omega_2} = \omega' \wedge \theta$  with $\omega' := \omega_1 - \omega_2$, i.e.\ $0 = \omega'(X_1)\theta(X_2) - \omega'(X_2)\theta(X_1)$.\\
We have $\omega' \in \mathfrak{A}^1_{\operatorname{hor}}(P, \mathfrak{g})^G \cong \Gamma(M, \underline{(\mathbb{R}^m)^\ast \tenr \mathbb{R}^m \tenr (\mathbb{R}^m)^\ast})$ and $\theta \in \mathfrak{A}^1_{\operatorname{hor}}(P, \mathbb{R}^m)^G \cong \Gamma(M, \underline{(\mathbb{R}^m)^\ast \tenr \mathbb{R}^m})$ since $\mathfrak{g} \subset (\mathbb{R}^m)^\ast \tenr \mathbb{R}^m$. Thus, $\omega' \wedge\theta$ corresponds precisely to an element of $\Gamma(M, \underline{\Lambda^2(\mathbb{R}^m)^\ast \tenr \mathbb{R}^m})$, obtained by wedging the two $(\mathbb{R}^m)^\ast$-factors in $\omega'$. With the $G$-equivariant homomorphism
\begin{align*}
    \iota\colon (\mathbb{R}^m)^\ast\otimes \mathbb{R}^m\otimes (\mathbb{R}^m)^\ast &\rightarrow \Lambda^2 (\mathbb{R}^m)^\ast\otimes \mathbb{R}^m\\
    \alpha_1\otimes v\otimes \alpha_2 &\mapsto (\alpha_1\wedge\alpha_2)\otimes v,
\end{align*}
it follows that $\omega' \in \kernel(\underline{\iota}) = \Gamma(M, \underline{\mathfrak{g}^1})$, where 
\begin{align*}
    \underline{\iota}\colon \Gamma(M, \underline{(\mathbb{R}^m)^\ast \tenr \mathfrak{g}}) \rightarrow \Gamma(M, \underline{\Lambda^2(\mathbb{R}^m)^\ast \tenr \mathbb{R}^m}),
\end{align*}
and $\mathfrak{g}^1 := \kernel(\iota_{\vert (\mathbb{R}^m)^\ast \tenr \mathfrak{g}})$ is called the \textbf{first prolongation} of $G$. Thus, knowledge of $\mathfrak{g}^1$ enables more precise computations involving differences of torsion-free connection forms.\\

From now on, we fix $G := \glsp{n}$ with Lie algebra $\mathfrak{g}$ and set $\widetilde{G} := \glxsp{n}$, so in particular $m = 4n$. Unless otherwise stated, maps are extended $\mathbb{C}$-linearly.\\
For a complex representation $\rho\colon G \rightarrow \Aut(V)$, we consider the map
\begin{align*}
    (\mathbb{R}^{4n} \tenr \mathbb{C})^\ast \tenc (\mathfrak{g} \tenr \mathbb{C}) \tenc V &\rightarrow (\mathbb{R}^{4n} \tenr \mathbb{C})^\ast \tenc V\\ 
    \alpha \tenc X \tenc v &\mapsto \alpha \tenc (T_e \rho(X)v),
\end{align*}
and we denote by $\sigma$ the restriction of this map to 
\begin{align*}
    (\mathfrak{g}^1 \tenr \mathbb{R}) \tenc V \subset (\mathbb{R}^{4n} \tenr \mathbb{C})^\ast \tenc (\mathfrak{g} \tenr \mathbb{C}) \tenc V.
\end{align*}
\begin{proposition}[{\cite[Proposition~5.1]{salamonDGQM}}]
    Let $M$ be a quaternionic manifold, $P$ the associated $G$-structure, and $V, W$ complex representations of $G$. For a surjective $G$\nobreakdash-equivariant homomorphism $\kappa\colon (\R^n\tenr\C)^\ast\tenc V \rightarrow W$, the operator $\underline{\kappa} \circ \nabla^\omega$ is independent of the choice of a torsion-free connection form $\omega \in \mathcal{C}(P)$ if and only if $\image(\sigma) \subset \kernel(\kappa)$.
\end{proposition}
We now aim to explicitly determine $\frg^1 \tenr \C$, as accomplished by Salamon in \cite{salamonPhD}. Therefore, we have to recall several results from \cite{salamonPhD} and \cite{salamonDGQM}. We have
\begin{align*}
    \frg = \mathfrak{gl} (n,\mathbb{H})\oplus_\R \mathfrak{sp} (1) = \R\cdot \Id_{\C^{2n}}\oplus \mathfrak{sl} (n,\mathbb{H})\oplus_\R \mathfrak{sp} (1)
\end{align*}
with $\mathfrak{gl} (n,\mathbb{H}) = \left\{ \begin{pmatrix}
        A&-\overline{B}\\B&\overline{A}
    \end{pmatrix} \mid A,B\in \C^{n\times n}\right\}$ and $\mathfrak{sl} (n,\mathbb{H}) = \lbrace T\in \mathfrak{gl} (n,\mathbb{H}) \mid \Tr_\C (T)=0\rbrace$.
Consider the representation $\rho \colon G\rightarrow \Aut (E\tenc H)$ and extend $T_e\rho$ to a $\C$-linear map $\rho \colon \frg\tenr \C\rightarrow \End (E\tenc H)$, which is injective.\\
Due to Remark~\ref{quatlabel29} and by grouping the $E$ (or its dual) and $H$ factors, we have
\begin{align*}
    &\End_\C (E\otimes_\C H)=(E\otimes_\C H)\otimes (E\otimes_\C H)^\ast\cong (E\otimes_\C H)\otimes (E^\ast\otimes_\C H^\ast)\\
    \cong & (E\otimes_\C H)\otimes (E^\ast\otimes_\C H) \cong (E\otimes_\C E^\ast )\otimes_\C (H\otimes_\C H) \\
    =& (E\otimes_\C E^\ast )\otimes_\C (\underbrace{\Lambda^2 H}_{=\C\cdot \omega_H}\oplus_\C \Sym^2 H)
\end{align*}
and
\begin{align}
\label{quatlabel32}
    T_e\rho (\frg\otimes_\R\C)= ((E\otimes_\C E^\ast )\otimes_\C \C\cdot \omega_H)\oplus_\C (\C\cdot \Id_E\otimes_\C \Sym^2 H).
\end{align}
We note that $T_e\rho$ is a $G$-equivariant map and that $\omega_H$ corresponds to $-\Id_H$ with respect to the isomorphism $H\otimes_\C H\cong H\otimes_\C H^\ast=\End(H)$.\\
In a similar way, due to Remark~\ref{quatlabel30} and the embedding (\ref{quatlabel32}), we obtain 
\begin{align}
    &(\R^{4n}\otimes_\R\C)^\ast\otimes_\C (\frg\otimes_\R \C)\notag\\
    \cong & (E^\ast\otimes_\C H)\otimes_\C (((E\otimes_\C E^\ast )\otimes_\C \C\cdot \omega_H)\oplus_\C (\C\cdot \Id_E\otimes_\C \Sym^2 H)) \label{quatlabel3}\\
    \cong & (E^\ast\tenc E\tenc E^\ast \tenc H\tenc \C\cdot \omega_H) \oplus_\C (E^\ast\tenc \C\cdot \Id_{E} \tenc H\tenc \Sym^2 H)\nonumber\\
    \cong & (\Lambda^2 E^\ast\tenc E\tenc H\tenc \C\cdot \omega_H) \oplus_\C (\Sym^2 E^\ast\tenc E\tenc H\tenc \C\cdot \omega_H) \notag\\
    &\oplus_\C (E^\ast\tenc \C\cdot \Id_{E} \tenc H\tenc \Sym^2 H)\label{quatlabel4},
\end{align}
since $E^\ast\tenc E\tenc E^\ast \cong E^\ast\tenc E^\ast \tenc E=(\Lambda^2 E^\ast \oplus_\C \Sym^2 E^\ast )\tenc E$. We now want to further decompose this into irreducible representations. To do so, we use \cite[Proposition~II.4.14]{broetomdieck} and therefore consider the representations of the groups $\GL (n,\mathbb{H})$ and $\spn{n}$ separately at first.\\
Consider the $\GL (n,\mathbb{H})$-equivariant contraction
\begin{align*}
    \mu \colon E^\ast\tenc E^\ast \tenc E &\rightarrow E^\ast\\
    \alpha_1\tenc \alpha_2\tenc v &\mapsto \alpha_2(v)\cdot\alpha_1
\end{align*}
and the restrictions of $\mu$ to $\Lambda^2 E^\ast \tenc E$ and $\Sym^2 E^\ast \tenc E$. Let
\begin{align*}
    C :=& \kernel (\mu_{\vert \Lambda^2 E^\ast \tenc E}),\\
    D :=& \kernel (\mu_{\vert \Sym^2 E^\ast \tenc E}),
\end{align*}
which are irreducible representations of $\GL (n,\mathbb{H})$ as mentioned in \cite[p.~36]{salamonDGQM}. In the same way, we obtain an isomorphism $E^\ast \tenc \C\cdot \Id_E\rightarrow E^\ast , \alpha\tenc \Id_E \mapsto 2n\alpha$.\\
This provides the decompositions $\Lambda^2 E^\ast \tenc E\cong E^\ast \oplus_\C C$ and $\Sym^2 E^\ast \tenc E \cong E^\ast \oplus_\C D$ into irreducible representations.\\
For the factors involving the space $H$, we consider the isomorphism 
\begin{align*}
    H\tenc \C\cdot \omega_H &\rightarrow H\\
    h\tenc \omega_H &\mapsto h
\end{align*}
and the decomposition $H\tenc \Sym^2 H \cong H\oplus_\C \Sym^3H$ (via contraction with $\omega_{H^\ast}\in \Lambda^2H^\ast$ or symmetrisation, due to the symmetry of the tensors, it is irrelevant how the contraction is applied), i.e.\ we have a map
\begin{align*}
    H\tenc \Sym^2 H &\rightarrow H\\
    h_1\tenc h_2\cdot h_3 &\mapsto \omega_{H^\ast} (h_1,h_2)h_3 + \omega_{H^\ast} (h_1,h_3)h_2.
\end{align*}
For $m\in \mathbb{N}$ and a vector space $V$, we write $mV:=\bigoplus\limits_{k=1}^m V$. Next, we summarise the results and obtain the following decomposition into irreducible representations
\begin{align*}
    &(\R^{4n}\otimes_\R\C)^\ast\otimes_\C (\frg\otimes_\R \C) \\
    \cong& (E^\ast \tenc H)\oplus_\C (C\tenc H) \oplus_\C (E^\ast \tenc H) \oplus_\C (D \tenc H) \\
    &\oplus_\C (E^\ast \tenc H) \oplus_\C (E^\ast \tenc \Sym^3 H)\\
    =& 3(E^\ast \tenc H) \oplus_\C (C\tenc H) \oplus_\C (D \tenc H) \oplus_\C (E^\ast \tenc \Sym^3 H),
\end{align*} where the ordering of the summands in (\ref{quatlabel4}) has been preserved, i.e.\ the ordering of the triples from $3(E^\ast\tenc H)$ corresponds exactly to the ordering of their appearance in (\ref{quatlabel4}).\\
The irreducible summand $E^\ast \tenc H$ is the only summand in the decomposition with a multiplicity greater than $1$.\\

We use an analogous approach to decompose $\Lambda^2 (\R^{4n}\tenr \C)^\ast \tenc (\R^{4n}\tenr \C)$ into irreducible representations. Here, we also have a decomposition $\Sym^2H\tenc H = H\oplus_\C \Sym^3 H$ with the help of the map
\begin{align*}
    \Sym^2 H\tenc H &\rightarrow H\\
    h_1\cdot h_2\tenc h_3 &\mapsto \omega_{H^\ast} (h_1,h_3)h_2 + \omega_{H^\ast} (h_2,h_3)h_1.
\end{align*}
Due to Remark~\ref{quatlabel30}, we therefore obtain
\begin{align*}
    &\Lambda^2 (\R^{4n}\tenr \C)^\ast \tenc (\R^{4n}\tenr \C) \cong  \Lambda^2 (E^\ast\tenc H) \tenc (E\tenr H)\\
    \cong & ((\Lambda^2 E^\ast \tenc \Sym^2 H) \oplus_\C (\Sym^2 E^\ast \tenc \Lambda^2 H))\tenc (E\tenc H)\\
    \cong & (\Lambda^2 E^\ast \tenc E \tenc \Sym^2 H\tenc H) \oplus_\C (\Sym^2 E^\ast\tenc E \tenc \C \cdot \omega_H \tenc H)\\
    \cong & (E^\ast \tenc H) \oplus_\C (C \tenc H) \oplus_\C (E^\ast \tenc \Sym^3 H) \\
    & \oplus_\C (C\tenc \Sym^3 H)\oplus_\C (E^\ast \tenc H) \oplus_\C (D \tenc H)\\
    =& 2(E^\ast\tenc H)\oplus_\C (C \tenc H) \oplus_\C (E^\ast \tenc \Sym^3 H) \oplus_\C (C\tenc \Sym^3 H)\oplus_\C (D \tenc H).
\end{align*}
As with the decomposition of $(\R^{4n}\otimes_\R\C)^\ast\otimes_\C (\frg\otimes_\R \C)$, we also preserve the ordering of the summands here.\\
With regard to the isomorphisms above, we now need to determine the kernel of the map
\begin{align*}
    \iota\colon (E^\ast \tenc H)\tenc (E\tenc H)\tenc (E^\ast \tenc H)&\rightarrow \Lambda^2 (E^\ast \tenc H)\tenc (E\tenc H)\\
    (\alpha_1\tenc h_1)\tenc (v\tenc h)\tenc (\alpha_2\tenc h_2)&\mapsto (\alpha_1\tenc h_1)\wedge (\alpha_2\tenc h_2)\tenc (v\tenc h)
\end{align*}
restricted to $(\R^{4n}\otimes_\R\C)^\ast\otimes_\C (\frg\otimes_\R \C)$ to obtain $\frg^1\tenr \C$.
\begin{remark}
    With regard to the reordering $\Lambda^2 (E^\ast\tenc H) \cong (\Lambda^2 E^\ast \tenc \Sym^2 H) \oplus_\C (\Sym^2 E^\ast \tenc \Lambda^2 H)$, $(\alpha_1\tenc h_1)\wedge (\alpha_2\tenc h_2)$ corresponds to $\frac{1}{2}(\alpha_1\wedge\alpha_2)\tenc (h_1\cdot h_2)+\frac{1}{2}(\alpha_1\cdot\alpha_2)\tenc (h_1\wedge h_2)$.
\end{remark}
\noindent To do this, we first consider the summands in the decomposition of $(\R^{4n}\otimes_\R\C)^\ast\otimes_\C (\frg\otimes_\R \C)$ with multiplicity $1$. For each of these summands, it is very easy to find an element that is not mapped to $0$ by $\iota$. By Schur's lemma \cite[Theorem~II.1.10]{broetomdieck}, it follows that $\iota$ is injective on this summand and maps only to an isomorphic summand in $\Lambda^2 (\R^{4n}\tenr \C)^\ast \tenc (\R^{4n}\tenr \C)$.\\
This leaves only the summand $E^\ast \tenc H$ with multiplicity $3$ to be considered:\\
Let $h,\widetilde{h}$ be a standard basis of $H$, i.e.\ $\omega_{H^\ast}(h,\widetilde{h})=1$ and $\widetilde{h}=J_r(h)$, and let $e_1,\ldots,e _{2n}$ be a basis of $E$ with dual basis $e^1,\ldots,e^{2n}\in E^\ast$. In particular, we then have $\omega_H = h\wedge\widetilde{h}$. We now define
\begin{align*}
\beta_1 := & (e^1\tenc h) \tenc (\Id_E \tenc \omega_H) \in (E^\ast\tenc H)\tenc (\R\cdot \Id_{\C^{2n}}\tenr \C),\\
\beta_2 := & (\sum\limits_{k=1}^{2n} (e^k\tenc h)\tenc (e_k\tenc e^1\tenc \omega_H))-\frac{1}{2n}\beta_1 \in (E^\ast\tenc H)\tenc (\mathfrak{sl}(n,\mathbb{H})\tenr\C),\\
\beta_3 := & (e^1\tenc \widetilde{h})\tenc (\Id_E\tenc h\cdot h)\\
&-(e^1\tenc h)\tenc (\Id_E\tenc h\cdot \widetilde{h})\in (E^\ast\tenc H)\tenc (\mathfrak{sp} (1)\tenr \C),
\end{align*}
where the elements have the form and ordering of elements of space (\ref{quatlabel3}).\\
With respect to the decomposition above, we have:
    \begin{align*}
        \beta_1 &\text{ corresponds to } \left(\frac{1}{2}(2n-1)e^1\tenc h , \frac{1}{2}(2n+1)e^1\tenc h , 0\right),\\
        \beta_2 &\text{ corresponds to } \left(-\frac{1}{2}(2n-\frac{1}{2n})e^1\tenc h , \frac{1}{2}(2n-\frac{1}{2n})e^1\tenc h , 0\right),\\
        \beta_3 &\text{ corresponds to } \left(0 , 0 , -6ne^1\tenc h\right).
    \end{align*}
Thus, $\langle \beta_1,\beta_2,\beta_3\rangle_\C=3\C\cdot (e^1\tenc h)\subset 3(E^\ast\tenc H)$. By replacing $e^1$ with $e^j$ and/or swapping $h$ and $\widetilde{h}$ in the definition of $\beta_1,\beta_2,\beta_3$, we obtain the entire space $3(E^\ast\tenc H)$.
\begin{lemma}[{\cite[Lemma~2.2]{salamonPhD}}]
\label{quatlabel5}
    We have
    \begin{align}
    \label{quatlabel6}
        \beta:=(n+1)\beta_1+2n\beta_2+n\beta_3\in \kernel (\iota_{\vert (\R^{4n}\otimes_\R\C)^\ast\otimes_\C (\frg\otimes_\R \C)}) = \frg^1\tenr \C
    \end{align}
    and thus, in particular, it follows that $\frg^1\tenr \C \cong E^\ast \tenc H$ as representations of $G$. As a result, $\frg^1\tenr \C$ is an irreducible representation.
\end{lemma}
We will need $(n+1)\beta_1+2n\beta_2+n\beta_3$ later on as an exemplary element of $\frg^1\tenr \C$ because Schur's lemma \cite[Theorem~II.1.10]{broetomdieck} can be used to easily prove properties of certain equivariant maps.
\subsection{Twistor Space of a Quaternionic Manifold} \label{quatsect3.3}
In this section, we review the definition of the twistor space and the relevant results from \cite{salamonPhD}. From now on, let $M$ always be a quaternionic manifold with its $G$-structure $P$ and a fixed torsion-free connection on $P$ which provides a quaternionic structure on $TM$.\\
The subset $Z:=\lbrace p\circ R_\nu\circ p^{-1}\mid p\in P, \nu\in Z_\mathbb{H}\rbrace\subset Q$ is called the \textbf{twistor space} of $M$. We denote the projection by $\pi_Z\colon Z\rightarrow M$. Thus, $Z_x:=\pi_Z^{-1} (x)$ consists of complex structures on $T_xM$. Instead of $z\in Z$, we often write $J_z$ for the corresponding complex structure on $T_{\pi_Z(z)}M$.\\
Consider the eigenvectors of elements from $Z_\mathbb{H}\subset\End(H)$. We obtain a diffeomorphism $Z\rightarrow\underline{\mathbf {P} (H\tenc\phi^{-1})}\linebreak\cong \underline{\mathbf{P}H}$, which maps $p\circ R_\nu\circ p^{-1}$ to $[p,[h\tenc 1]]$ for an eigenvector $h\neq 0$ of $\nu$ to the eigenvalue $i$, where $G$ acts on $\mathbf {P} (H\tenc\phi^{-1})$ by $[T,\lambda]\cdot [h\tenc 1]:=[(\det_\C(T))^{-\frac{1}{2(n+1)}}\lambda (h)\tenc 1]$ (analogously on $\mathbf{P} H$).
\subsection*{Complex structure on $Z$}
A torsion-free connection on $P$ induces a decomposition $TZ=T^VZ\oplus T^HZ$. Here, ${T\pi_Z}_{\vert T^HZ}\colon T_z^HZ\rightarrow T_{\pi_Z(z)}M$ is an isomorphism, and since the fibre $Z_x\cong \mathbf {P} (H\tenc\phi^{-1})$ over every point $x\in M$ is a complex manifold with its tangent space $T^VZ_{\vert Z_x}$, we can now define an almost complex structure on $TZ$, which, according to \cite{salamonPhD} and \cite{salamonDGQM}, is indeed integrable and independent of the choice of connection on $P$. We denote this complex structure by $\mathcal{J}\in\End(TZ)$. On $T^VZ$, $\mathcal{J}$ is given by the complex structure of the fibres, i.e.\ $\mathbf{P}(H\tenc\phi^{-1})$, and on $T^HZ$ we have $\mathcal{J}_z = ({T\pi_Z}_ {\vert T^HZ})^{-1}\circ J_z\circ {T\pi_Z}_{\vert T^HZ}$.
\subsection*{Holomorphic vector bundles}
The action of $G$ on powers of the tautological line bundle $\mathcal{O}(-r)$ of $\mathbf{P}H$ or its dual $\mathcal{O} (r)$, which is defined by $[T,\lambda]\cdot ([h],yh^{\otimes r}): =([\lambda(h)],y(\lambda h)^{\otimes r})$ for $([h],yh^{\otimes r})\in\mathcal{O}(-r)$ and $[T,\lambda]\cdot s:=s\circ (\lambda^{-1})^ {\otimes r}$ for $s\in\mathcal{O} (r)$, is only well-defined for $r\in2\cdot\N_0$. Thus, for $r\in 2\cdot\N_0$, there exists a globally defined line bundle $L^{-r}:=\underline{\mathcal{O}(-r)}$ and its dual $L^r := \underline{\mathcal{O} (r)}$. However, due to the construction of the complex structure of $Z$, these two line bundles are generally not holomorphic line bundles on the whole of $Z$. Yet, $\mathcal{L}^{-r}:=L^{-r}\tenc \pi_Z^\ast\underline{\phi^{-r}}$ and its dual $\mathcal{L}^r:=L^r\tenc \pi_Z^\ast\underline{\phi^r}$ are always holomorphic line bundles. We regard $\mathcal{L}^{-r}$ as a subspace of $\underline{\mathcal{O}(-r)\tenc \phi^{-r}}$ and $\mathcal{L}^r$ as a subspace of $\underline{\mathcal{O}(r)\tenc \phi^r}$. For $s\in\Gamma^{\operatorname{hol}}(\mathbf{P}H,\mathcal{O}(r))$, $\underline{s}_{[p,[h]]}: =[p,([h],s_{[h]}\tenc 1)]$ is a local holomorphic section in $\mathcal{L}^r$.\\

For $J_z\in Z_x$, $z=[p,[h]]$, we denote by $T^{1,0}_xM_{J_z}$ the eigenspace of $J_z$ on $T_xM\tenr\C$ corresponding to the eigenvalue $i$ and by $T^{0,1} _xM_{J_z}$ the eigenspace corresponding to the eigenvalue $-i$. With $TM\tenr\C\cong \underline{E\tenc H}$ and $T^\ast M\tenr\C\cong \underline{E^\ast\tenc H}$, we get
 \begin{align*}
        (T^{1,0}_x M_{J_z})^\ast &\cong \lbrace [p,\eta\otimes h]\mid p\in P,\eta\in E^\ast\rbrace,\\
        (T^{0,1}_xM_{J_z})^\ast &\cong \lbrace [p,\eta\otimes J_r(h)]\mid p\in P,\eta\in E^\ast\rbrace,
\end{align*}
and
\begin{align}
        T^{1,0}_xM_{J_z} &\cong \lbrace [p,e\otimes J_r(h)]\mid p\in P,e\in E\rbrace,\notag\\
        T^{0,1}_xM_{J_z} &\cong \lbrace [p,e\otimes h]\mid p\in P,e\in E\rbrace . \label{quatlabel7}
\end{align}
We define the subvector bundle $B_2\subset \Lambda^2T^\ast M\tenr\C$ by $B_{2,x}:=\bigcap_{J_z\in Z_x}\Lambda^{1,1}T^{\ast}_xM_{J_z}\subset \Lambda^2 T^\ast_x M\otimes_\R \C$ for every $x\in M$. 
\begin{definition}
    We call a complex vector bundle $\pi_F\colon F\rightarrow M$ together with a connection $\nabla^F$ a \textbf{$B_2$-vector bundle} if the curvature of $\nabla^F$ only has form components in $B_2$.
\end{definition}
Due to $\pi_Z^\ast\colon \Lambda^{p,q}T^\ast_xM_{J_z}\rightarrow \Lambda^{p,q}T^\ast_zZ$, the next result follows from \cite[Proposition~1.3.7]{kobayashiDGMCVB}:
\begin{theorem}[{\cite[Theorem~5.3]{salamonQM}}]
    For a $B_2$-vector bundle $\pi_F:F\rightarrow M$ with connection $\nabla^F$, $\pi_Z^\ast F\rightarrow Z$ with the pullback connection is a holomorphic vector bundle.
\end{theorem}
The complex structure of $Z$ and its holomorphic vector bundles will be very helpful in the later sections.
\begin{remark}
    With the decomposition $\Lambda^2 (E^\ast\tenc H)=(\Sym^2E^\ast \tenc \Lambda^2H)\oplus(\Lambda^2E^\ast\tenc \Sym^2H)\cong (\Sym^2E^\ast)\oplus(\Lambda^2E^\ast\tenc \Sym^2H)$, we get $B_2\cong\underline{\Sym^2E^\ast}$.
\end{remark}
\begin{example}[\cite{salamonQM}]
    The twistor space of $M=\mathbf{P}^n\mathbb{H}$ is given by $Z=\mathbf{P}^n\mathbb{C}$, where
    \begin{align*}
        \pi_Z\colon \mathbf{P}^n\mathbb{C}&\rightarrow \mathbf{P}^n\mathbb{H}\\
        \left[\begin{pmatrix}
            u\\v
        \end{pmatrix}\right]_\C &\mapsto \left[\begin{pmatrix}
            u\\v
        \end{pmatrix}\right]_\mathbb{H}
    \end{align*}
    together with the identification $\mathbb{H}^n\rightarrow \C^{2n}$.
\end{example}
\subsection{Salamon Complex} \label{quatsect3.4}
With the decomposition into subrepresentations 
\begin{align*}
    \Lambda^k (E^\ast\otimes_\C H)=\bigoplus\limits_{0\leq b\leq a\leq 2n,a+b=k} \Lambda^{a,b}E^\ast \otimes_\C \Sym^{a-b}H
\end{align*}
from \cite{kkquattorsion}, where $\Lambda^{a,b}E^\ast\subset \Lambda^{a}E^\ast\tenc\Lambda^{b}E^\ast$ is the kernel of a certain $\GL (n,\mathbb{H})$-equivariant map, one obtains a decomposition 
\begin{align*}
    \Lambda^k(T^\ast M\otimes_\R\C)\cong\bigoplus\limits_{0\leq b\leq a\leq 2n,a+b=k} \underline{\Lambda^{a,b}E^\ast \otimes_\C \Sym^{a-b}H}
\end{align*}
with $T^\ast M\otimes_\R \C\cong \underline{E^\ast \otimes_\C H}$. Here, canonically $\Lambda^{k,0}E^\ast\cong \Lambda^k E^\ast$ holds.
\begin{remark}
\label{quatlabel8}
    We have $\Lambda^{k,0}E^\ast \otimes_\C \Sym^{k}H\cong \Lambda^{k}E^\ast \otimes_\C \Sym^{k}H$ and for any $1$-form $\eta\otimes h\in E^\ast \otimes H$ and a $k$-form $(\eta_1\wedge\ldots\wedge \eta_k)\otimes_\C (h_1\cdot\ldots\cdot h_k)\in \Lambda^{k}E^\ast \otimes_\C \Sym^{k}H$, the wedge product is given by
    \begin{align*}
        (\eta\otimes_\C h)\wedge ((\eta_1\wedge\ldots\wedge \eta_k)\otimes_\C (h_1\cdot\ldots\cdot h_k))
        = &\frac{1}{k+1} (\eta\wedge\eta_1\wedge\ldots\wedge \eta_k)\otimes_\C (h\cdot h_1\cdot\ldots\cdot h_k)\\
        &+ (\text{terms outside }\Lambda^{k}E^\ast \otimes_\C \Sym^{k}H)
    \end{align*}
\end{remark}
Together with \cite[Theorem~3.3.4]{kkdiffgeoENG}, this remarks provides motivation for the following definition of the differential operator $\delta_r$ (see Remark~\ref{quatlabel9} below).\\
\hfill\break
Let $k\in \N_0$, $r\in 2\cdot\N_0$, we define
\begin{align*}
    V^{k,r}:= \underline{\Lambda^k E^\ast \otimes_\C \Sym^{k+r} H\otimes_\C\phi^r}
\end{align*}
to be the associated vector bundle to the $G$-representation $\Lambda^k E^\ast \otimes_\C \Sym^{k+r} H\otimes_\C\phi^r$.
\begin{proposition}[{\cite[Proposition~3.2]{salamonPhD}}]
    Let the $G$-equivariant homomorphism
     \begin{align*}
    \kappa \colon (E^\ast \otimes_\C H)\otimes_\C(\Lambda^s E^\ast \otimes_\C \Sym^{t} H\otimes_\C\phi^r) \rightarrow  \Lambda^{s+1} E^\ast \otimes_\C \Sym^{t+1} H \otimes_\C\phi^r   
    \end{align*}
    be defined with $s,t\in \mathbb{N}_0$, $r\in\R$ on basis elements as
    \begin{align*}
        &\kappa ((\eta\otimes_\C h)\otimes_\C(\eta_{1}\wedge\ldots \wedge \eta_{s})\otimes_\C (h_{1}\cdot\ldots\cdot h_{t})\otimes c) \\
        :=& \frac{1}{t+1}(\eta\wedge\eta_{1}\wedge\ldots \wedge \eta_{s})\otimes_\C (h\cdot h_{1}\cdot\ldots\cdot h_{t})\otimes_\C c.
    \end{align*}
    Then $\kappa$ is always surjective, but $\image (\sigma)\subset \kernel (\kappa)$ only holds if $r=t-s$.
\end{proposition}
\begin{proof}
We only cite an outline of the proof, which can be found in \cite{salamonPhD}, and omit the complete calculations. Consider the element $\beta\in \frg^1\tenr \C$ from (\ref{quatlabel6}) and show that for all $v\in \Lambda^s E^\ast \otimes_\C \Sym^{t} H\otimes_\C\phi^r$, that $\kappa\circ \sigma(\beta\tenc v)=0$ holds if and only if $r=t-s$.\\
According to Schur's lemma, due to the irreducibility of $\frg^1\tenr \C$ and $\Lambda^s E^\ast \otimes_\C \Sym^{t} H\otimes_\C\phi^r$, it follows that $\kappa\circ \sigma(\cdot\tenc v)\colon \frg^1\tenr \C\rightarrow \Lambda^{s+1} E^\ast \otimes_\C \Sym^{t+1} H\otimes_\C\phi^r$ is constant $0$.
\end{proof}
\begin{proposition}[{\cite[Proposition~3.3]{salamonPhD}}]
    \label{quatlabel25}
    Let $\nabla$ be the covariant connection on $V^{k,r}$ for every $k\in \N_0$, $r\in 2\cdot\N_0$, which is induced by a torsion-free connection of the $G$-structure $P$ on the associated vector bundle.\\
    This yields a differential operator 
    \begin{align*}
        \delta_r:=\underline{\kappa}\circ\nabla\colon \Gamma (M,V^{k,r})\stackrel{\nabla}{\rightarrow}\Gamma (M, \underline{E^\ast\otimes H}\otimes_\C V^{k,r})\stackrel{\underline{\kappa}}{\rightarrow}\Gamma (M,V^{k+1,r}),
    \end{align*}
    which provides an elliptic complex
    \begin{align*}
        0\rightarrow \Gamma(M,V^{0,r})\stackrel{\delta_r}{\rightarrow}\ldots\stackrel{\delta_r}{\rightarrow}\Gamma(M,V^{2n,r})\stackrel{\delta_r}{\rightarrow} 0.
    \end{align*}
    We call this complex the \textbf{Salamon complex}.\\
    Let $F\rightarrow M$ be a $B_2$-vector bundle with a connection $\nabla^F$. With
    \begin{align*}
        \nabla\otimes 1+1\otimes \nabla^F \colon\Gamma(M,V^{k,r}\otimes_\C F)\rightarrow \Gamma(M,\underline{E^\ast\otimes_\C H}\otimes_\C V^{k,r}\otimes_\C F),
    \end{align*}
    we similarly obtain a differential operator. We define
    \begin{align*}
        \delta_{F,r}:=\underline{\kappa}^F\circ (\nabla\otimes 1+1\otimes \nabla^F)
    \end{align*}
    with $\underline{\kappa}^F:= \underline{\kappa}\otimes_\C\Id_F$.
    The differential operator $\delta_{F,r}$ also yields an elliptic complex.\\
    The differential operators $\delta_r$ and $\delta_{F,r}$ are in both cases independent of the choice of torsion-free connection of $P\subset \GL (M)$.
\end{proposition}
\begin{remark}
    We have $\delta_{F,r}=\delta_r\tenc 1 + \underline{\kappa}^F\circ (1\otimes \nabla^F)$.
\end{remark}
\begin{remark}
    \label{quatlabel9}
    If $F$ is the trivial complex line bundle and $r=0$, then the Salamon complex is a quotient complex of the de Rham complex, since $\underline{\kappa}\circ \nabla$ is the de Rham operator followed by a projection onto a subspace (see Remark~\ref{quatlabel8}). We have $V^{k,0}_x=\sum\limits_{J_x\in Z_x} \Lambda^{k,0}T_x^\ast M_{J_x}\oplus \Lambda^{0,k}T_x^\ast M_{J_x}$ and its complementary space in $\Lambda^k (T^\ast M\tenr \C)$ is given by $\bigcap\limits_{J_x\in Z_x} \Lambda^{k-1,1}T_x^\ast M_{J_x}\oplus\ldots\oplus\Lambda^{1,k-1}T_x^\ast M_{J_x}$ according to \cite{salamonDGQM}.
\end{remark}
\begin{definition}
    We denote the cohomology of the Salamon complex for the differential operator $\delta_{F,r}$ by
    \begin{align*}
        H^{k}_{\delta_{F,r}}(M) := \dfrac{\kernel\left({\delta_{F,r}}_{\vert \Gamma (M,V^{k,r}\otimes_\C F)}\right)}{\image\left( {\delta_{F,r}}_{\vert \Gamma (M,V^{k-1,r}\otimes_\C F)}\right)}.
    \end{align*}
\end{definition}

\subsection{Relation of  the Salamon Complex and the Dolbeault Complex} \label{quatsect3.5}
Let $r\in 2\cdot\N_0$ and $\pi_F\colon F\rightarrow M$ be a $B_2$-vector bundle. We now repeat the results and proofs from \cite[Section~1.6]{salamonPhD} for the general case, since in \cite{salamonPhD} it was only proven for the special case of $r=0$ and the trivial line bundle $F=M\times \C$. The main result is that there is a canonical embedding $\Gamma(M,V^{k,r}\tenc F)\hookrightarrow\mathfrak{A}^{0,k} (Z,\mathcal{L}^r\tenc \pi_Z^\ast F)$ such that $\delta_{F,r}$ is induced by the Dolbeault operator $\dolb$ on $\mathfrak{A}^{0,\bullet}(Z,\mathcal{L}^r\tenc \pi_Z^\ast F)$. The embedding induces an isomorphism between the cohomologies $H^{\bullet}_{\delta_{F,r}}(M)$ and $H^{\bullet}_{\dolb}(Z,\mathcal{L}^r\tenc \pi_Z^\ast F)$.\\

In the following section, we define all the necessary maps and spaces in order to obtain the commutative diagram in Figure~\ref{quatlabel10}.
\begin{figure}[htbp]
  \centering
   \begin{tikzcd}[row sep=large, column sep = small]
    0\arrow[d] & & 0\arrow [d] & \\
    \Gamma (M, V^{0,r}\otimes F)\arrow{d}{\delta_{F,r}}\arrow{r}{\mu_{F,r}} & \Gamma (Z, W^{0,r}\otimes \pi_Z^\ast F)\arrow[swap, near start]{drr}{\dolb^V}\arrow{r}{\beta_{F,r}} & \mathfrak{A}^{0,0}(Z,\mathcal{L}^r\otimes\pi_Z^\ast F)\arrow[near start, crossing over]{d}{\dolb} &\\
    \Gamma (M, V^{1,r} \otimes F)\arrow{d}{\delta_{F,r}}\arrow{r}{\mu_{F,r}} & \Gamma (Z, W^{1,r}\otimes \pi_Z^\ast F)\arrow{r}{\beta_{F,r}} \arrow[swap,near start]{drr}{\dolb^V} & \mathfrak{A}^{0,1}(Z,\mathcal{L}^r\otimes\pi_Z^\ast F)\arrow[swap]{r}{\proj^V} \arrow[near start, crossing over]{d}{\dolb} & \mathfrak{A}^{0,1}(Z,\mathcal{L}^r\otimes\pi_Z^\ast F)^V\\
    \textover[c]{\vdots}{$\Gamma (M, V^{1,r}\otimes F)$}\arrow{d}{\delta_{F,r}}\arrow{r}{\mu_{F,r}} & \textover[c]{\vdots}{$\Gamma (Z, W^{1,r}\otimes\pi_Z^\ast F)$}\arrow{r}{\beta_{F,r}}\arrow[swap, near start]{drr}{\dolb^V} & \textover[c]{\vdots}{$\Gamma (Z,\pi_Z^\ast F\otimes W^{1,r})$}\arrow[swap]{r}{\proj^V} \arrow[near start,crossing over]{d}{\dolb} & \textover[c]{\vdots}{$\mathfrak{A}^{0,1}(Z,\mathcal{L}^r\otimes\pi_Z^\ast F)^V$}\\
    \Gamma (M, V^{2n,r}\otimes F)\arrow{d}{\delta_{F,r}}\arrow{r}{\mu_{F,r}} & \Gamma (Z, W^{2n,r}\otimes\pi_Z^\ast F)\arrow{r}{\beta_{F,r}}\arrow[swap, near start]{drr}{\dolb^V} & \mathfrak{A}^{0,2n}(Z,\mathcal{L}^r\otimes\pi_Z^\ast F)\arrow[swap]{r}{\proj^V} \arrow[near start,crossing over]{d}{\dolb} & \mathfrak{A}^{0,2n}(Z,\mathcal{L}^r\otimes\pi_Z^\ast F)^V\\
    0 &  & \mathfrak{A}^{0,2n+1}(Z,\mathcal{L}^r\otimes\pi_Z^\ast F)\arrow{d}{\dolb}\arrow[swap]{r}{\proj^V} & \mathfrak{A}^{0,2n+1}(Z,\mathcal{L}^r\otimes\pi_Z^\ast F)^V\\
    & & 0 &
    \end{tikzcd}
\caption{commutative diagram}
\label{quatlabel10}
\end{figure}
Here, $\beta_{F,r}\circ\mu_{F,r}\colon \Gamma(M,V^{\bullet,r}\tenc F)\hookrightarrow\mathfrak{A}^{0,\bullet}(Z,\mathcal{L}^r\tenc \pi_Z^\ast F)$ is the embedding mentioned above. For the maps $\mu_{F,r}$ and $\beta_{F,r}$, we omit $F$ and $r$ from the index if $F$ is the trivial line bundle or $r=0$, respectively.\\

In the following, we often only need to work locally, so we use the fact that locally, on sufficiently small neighbourhoods $U\subset M$ of a point $x\in M$, there is always a principal bundle $\widetilde{P}$ with structure group $\widetilde{G}=\glxsp{n}$ as a $2$-fold cover of $P$, so that
\[
\begin{tikzcd}
\widetilde{P}\times \widetilde{G} \arrow[r] \arrow[d] & \widetilde{P} \arrow[d,"\zeta"]\arrow[dr]& \\
P_{\vert U}\times G \arrow[r] & P_{\vert U}\arrow[r] &M
\end{tikzcd}
\]
 commutes, since $\widetilde{G}\rightarrow G$ is a $2$-fold cover. The connection on $P$ thus also provides a connection on $\widetilde{P}$. Although we only work locally, the constructions and maps are still canonical, i.e.\ independent of the local trivialisation.\\
Let $U\subset M$ be a neighbourhood of $x\in M$ such that a principal bundle $\widetilde{P}$ as described above exists. Now, every local standard basis of $\widetilde{P}\times_{\widetilde{G}} H$ can be written as $[\widetilde{p},h_1],[\widetilde{p},h_2]$ for a standard basis $h_1 , h_2 := J_r(h_1)$ of $H$ and a local section $\widetilde{p}\colon U\rightarrow\widetilde{P}$.\\
Consider the sections $s_1',s_2'\in \Gamma^{\operatorname{hol}} (\mathbf{P}H,\mathcal{O} (1))$ with $s_{1, [h_1+zh_2]}' ([h_1+zh_2],h_1+zh_2)=1$ and $s_{2, [zh_1+h_2]}' ([zh_1+h_2],zh_1+h_2)=1$. Then $s_{1, [z_1h_1+z_2h_2]}'=\frac{z_1}{z_2}s_{2, [z_1h_1+z_2h_2]}' $ and $\frac{z_2}{z_1}s_{1, [z_1h_1+z_2h_2]}'=s_{2, [z_1h_1+z_2h_2]}'$ for $z_2\neq 0$ and $z_1\neq 0$, respectively. In particular, $s_{1, [h_2]}'=0$ and $s_{2, [h_1]}'=0$.
\subsection*{Construction of $\beta_{F,r}$}
The map
\begin{align*}
    \varepsilon \colon \overline{\mathcal{O}(-1)} &\rightarrow \mathcal{O} (1)\\
    ([z_1h_1+z_2h_2],c(z_1h_1+z_2h_2)) &\mapsto \overline{c} (\overline{z_1}s_{1, [z_1h_1+z_2h_2]}'+\overline{z_2}s_{2, [z_1h_1+z_2h_2]}')
\end{align*}
is a $\Sp (1)$-equivariant isomorphism of vector bundles.
For $z_2\neq 0$, we have
\begin{align*}
    \varepsilon ([z_1h_1+z_2h_2],\frac{z_1}{z_2}h_1+h_2)&=\varepsilon ([z_1h_1+z_2h_2],\frac{1}{z_2}(z_1h_1+z_2h_2)) \\
    &= \frac{1}{\overline{z_2}} (\overline{z_1}s_{1, [z_1h_1+z_2h_2]}'+\overline{z_2}s_{2, [z_1h_1+z_2h_2]}')\\
    &= \frac{\overline{z_1}}{\overline{z_2}}\underbrace{s_{1, [z_1h_1+z_2h_2]}'}_{=\frac{z_1}{z_2}{s'_{2, [z_1h_1+z_2h_2]}}}+s_{2, [z_1h_1+z_2h_2]}' \\
    &= (\frac{z_1\overline{z_1}}{z_2\overline{z_2}}+1)s_{2, [z_1h_1+z_2h_2]}'.
\end{align*}
For $h=z_1h_1+z_2h_2$, we get
\begin{align*}
    \varepsilon^{-1} (s'_{1, [h]})=([h],\frac{\overline{z_1}}{\| h\|^2}h), &&
    \varepsilon^{-1} (s'_{2, [h]})=([h],\frac{\overline{z_2}}{\| h\|^2}h)
\end{align*}
for $\| h\|^2=\vert z_1\vert^2 +\vert z_2\vert^2=z_1\overline{z_1}+z_2\overline{z_2}$ and therefore
\begin{align*}
    J_r(\varepsilon^{-1} (s'_{1, [h]}))=([h],\frac{z_1}{\| h\|^2}J_r(h)), &&
    J_r(\varepsilon^{-1} (s'_{2, [h]}))=([h],\frac{z_2}{\| h\|^2}J_r(h)).
\end{align*}
For $k\in\N_0$, this also yields an isomorphism $\varepsilon_k := \varepsilon^{\otimes k}\colon \overline{\mathcal{O}(-k)}\rightarrow \mathcal{O} (k)$.\\
We define 
\begin{align*}
    s^{\otimes l} s'^{\otimes k-l}:=\underbrace{s\otimes_\C\ldots \otimes_\C s}_{l\text{ times}} \otimes_\C\underbrace{s'\otimes_\C\ldots \otimes_\C s'}_{k-l\text{ times}}
\end{align*}
for $s,s'\in \mathcal{O} (1)$.\\
Then let ${s}_{1, [\zeta (\widetilde{p}),[h]]}:=[\widetilde{p},s'_{1, [h]}]$ and ${s}_{2, [\zeta (\widetilde{p}),[h]]}: =[\widetilde{p},s'_{2, [h]}]$ be the corresponding sections in the local vector bundle $L^1$ existing on $\pi_Z^{-1} (U)\subset Z$.\\
Since $\varepsilon_k \colon \overline{\mathcal{O}(-k)} \rightarrow \mathcal{O} (k)$ is a $\Sp (1)$-equivariant isomorphism, we can associate an isomorphism $\underline{\varepsilon_k} \colon \overline{L^{-k}}\rightarrow L^k$ to it.\\
Pointwise, this gives an isomorphism $(\widetilde{P}\times_{\widetilde{G}} H)_{ y}\cong \Gamma^{\operatorname{hol}} (\pi_Z^{-1} (y),L^1_{ \pi_Z^{-1} (y)})$ by defining
\begin{align*}
    {h}_{1, y}&\mapsto {-s}_{2, \pi_Z^{-1} (y)},\\
    {h}_{2, y}&\mapsto {s}_{1, \pi_Z^{-1} (y)}.
\end{align*}
For the local sections
\begin{align*}
    {t}_{1, [\zeta (\widetilde{p}),[z_1h_1+z_2h_2]]}:=&[\widetilde{p},([z_1h_1+z_2h_2],h_1+\frac{z_2}{z_1}h_2)],\\
    {t}_{2, [\zeta (\widetilde{p}),[z_1h_1+z_2h_2]]}:=&[\widetilde{p},([z_1h_1+z_2h_2],\frac{z_1}{z_2}h_1+h_2)]
\end{align*}
of $L^{-1}$, this implies
\begin{align*}
    \underline{\varepsilon} ({t}_{2, [\zeta (\widetilde{p}),[z_1h_1+z_2h_2]]}) &= \underline{\varepsilon} ([\widetilde{p},([z_1h_1+z_2h_2],\frac{z_1}{z_2}h_1+h_2)]) \\
    &= \underline{\varepsilon} ([\widetilde{p},([z_1h_1+z_2h_2],\frac{1}{z_2}(z_1h_1+z_2h_2))])\\
    &= [\widetilde{p},\frac{1}{\overline{z_2}} (\overline{z_1}s_{1, [z_1h_1+z_2h_2]}'+\overline{z_2}s_{2, [z_1h_1+z_2h_2]}') ] \\
    &= [\widetilde{p},\frac{\overline{z_1}}{\overline{z_2}}\underbrace{s_{1, [z_1h_1+z_2h_2]}'}_{=\frac{z_1}{z_2}s'_{2, [z_1h_1+z_2h_2]}}+s_{2, [z_1h_1+z_2h_2]}' ]\\
    &=[\widetilde{p},(\frac{z_1\overline{z_1}}{z_2\overline{z_2}}+1)s_{2, [z_1h_1+z_2h_2]}']
    = (\frac{z_1\overline{z_1}}{z_2\overline{z_2}}+1){s}_{2, [\zeta (\widetilde{p}),[z_1h_1+z_2h_2]]}.
\end{align*}
We denote by $T^{1,0,H}Z$ and $T^{1,0,V}Z$ the eigenspace of $\mathcal{J}\otimes_\R \Id_\C$ on $T^HZ\otimes_\R\C$ and $T^VZ\otimes_\R\C$, respectively, for the eigenvalue $i$. For $z=[p,[h]]$, we then have
\begin{align*}
    T^{\ast 1,0,H}_zZ:=(T^{1,0,H}_z Z)^\ast \cong \pi_Z^\ast( \lbrace [p,\eta\otimes h]\mid \eta\in E^\ast\rbrace ).
\end{align*}
Thus, $T^{\ast 1,0,H}_zZ$ corresponds exactly to the vector bundle $\pi_Z^\ast (\underline{E^\ast})\otimes L^{-1}\subset \pi_Z^\ast (\underline{E^\ast\otimes_\C H)}$, which therefore exists globally, even though the vector bundles $\pi_Z^\ast (\underline{E^\ast})$ and $L^{-1}$ only exist locally. Here, $\pi_Z^\ast$ only refers to the pullback of vector bundles, while the pullback of forms also allows us to regard $\pi_Z^\ast( \lbrace [p,\eta\otimes h]\mid \eta\in E^\ast\rbrace )$ as a subvector bundle of $(TZ\otimes_R\C)^\ast$.\\
This works as follows:
Let $E_1\subset TZ\otimes_\R\C$ and $E_2\subset TM\otimes_\R\C$ be subvector bundles such that
\begin{align*}
    T\pi_Z \colon E_1&\rightarrow \pi_Z^\ast E_2\\
    v_z &\mapsto (z,\underbrace{T_z \pi_Z (v_z)}_{\in T_{\pi_Z (z)}M})
\end{align*}
is an isomorphism (here: $E_2:=TZ^H\otimes_\R\C$ and $E_1:=TM\otimes_\R\C$.)
Then the pullback of forms is given by
\begin{align*}
    (T\pi_Z)^\ast \colon \underbrace{(\pi_Z^\ast E_2)^\ast}_{=\pi_Z^\ast E_2^\ast} &\rightarrow E_1^\ast \\
    (z,\alpha) &\mapsto \alpha \circ T_z\pi_Z,
\end{align*}
where $\alpha\in T^\ast_{\pi_Z (z)}M$. In particular, this pullback is an isomorphism. We usually work with $\pi_Z^\ast F^\ast$ instead of its image $E_1^\ast=(T\pi_Z)^\ast (\pi_Z^\ast E_2^\ast)$ in $(TZ\otimes_\R\C)^\ast$.\\
Now $0\rightarrow T^{\ast 1,0,H}Z \rightarrow T^{\ast 1,0}Z \rightarrow T^{\ast 1,0,V}Z\rightarrow 0$ is an exact sequence. By taking the $k$\nobreakdash-th exterior power, we also obtain an exact sequence with the isomorphism $T^{\ast 1,0,H}_zZ\cong\pi_Z^\ast( \lbrace [p,\eta\otimes h]\mid \eta\in E^\ast\rbrace )$:
\begin{align*}
    0\rightarrow \underbrace{\pi_Z^\ast (\underline{\Lambda^{k} E^\ast})\otimes_\C L^{-k}}_{\cong\Lambda^{k,0} T^{\ast H} Z} \rightarrow \Lambda^{k,0} T^\ast Z \rightarrow (\underbrace{\pi_Z^\ast (\underline{\Lambda^{k-1} E^\ast})\otimes_\C L^{-(k-1)}}_{\cong\Lambda^{k-1,0} T^{\ast H}Z})\otimes_\C (T^{\ast 1,0,V}Z)\rightarrow 0.
\end{align*}
The last space arises because of $\Lambda^k (V_1^\ast\oplus V_2^\ast) = (\Lambda^k V_1^\ast)\oplus (\Lambda^{k-1} V_1^\ast\otimes V_2^\ast)$ for $\dim V_2=1$, since $T^{\ast 1,0,V}Z$ is one-dimensional. If $k=0$, then let $\pi_Z^\ast (\underline{\Lambda^{k-1} E^\ast})\otimes_\C L^{-(k-1)}$ be the trivial line bundle. We will continue to use this notation throughout the rest of this section.\\
For a vector space with a complex structure, complex conjugation on the second factor in $V\tenr \C$ yields a $\C$-linear isomorphism $\overline{V^{1,0}}\rightarrow V^{0,1}$.\\
In this way, we obtain an exact sequence
\begin{align*}
    0\rightarrow \pi_Z^\ast (\underline{\Lambda^{k} E^\ast})\otimes_\C L^{k}\xrightarrow{\beta}\Lambda^{0,k} T^\ast Z \rightarrow (\pi_Z^\ast (\underline{\Lambda^{k-1} E^\ast})\otimes_\C L^{k-1})\otimes_\C (T^{\ast 0,1}Z^V)\rightarrow 0,
\end{align*}
by complex conjugation on the second factor of $(TZ\otimes_\R\C)^\ast=T^\ast Z\otimes_\R\C$ and with the isomorphism $\overline{L^{-k}}\rightarrow L^k$, where $\pi_Z^\ast (\underline{\Lambda^{k} E^\ast})\otimes_\C L^{k}\cong \Lambda^{0,k} T^\ast Z^H$ (also for $k-1$).\\
Let the map $\beta$ be defined for $\eta\in \Lambda^{k} E^\ast,s\in \mathcal{O} (k)_{[h]}$ by
\begin{align*}
    \beta ([p,([h],\eta\otimes_\C s]):= (T\pi_Z)^\ast (\underbrace{[p,([h],\eta\otimes_\C J_r^{\otimes k}(\varepsilon_k^{-1} (s)))]}_{\in \pi_Z^\ast (\underline{E^\ast\otimes_\C H})\cong \pi_Z^\ast (T^\ast M\otimes_\R\C)}).
\end{align*}
For $[p,([h],\eta\otimes_\C {s_2'}^{\otimes k}_{ [h]})]\in\pi_Z^\ast (\underline{\Lambda^{k} E^\ast})\otimes_\C L^{k}$ for $h=z_1h_1+z_2h_2$ with $z_2\neq 0$, we have
\begin{align*}
    \beta ([p,([h],\eta\otimes_\C {s_2'}^{\otimes k}_{ [h]})])&=(T\pi_Z)^\ast ( [p,([h],\eta\otimes_\C (J_r(\frac{\overline{z_2}}{\| h\|^2}h))^{\otimes k})]) \\
    &= (T\pi_Z)^\ast ( [p,([h],\frac{z_2^k}{\| h\|^{2k}} \eta\otimes_\C (J_r(h))^{\otimes k})]).
\end{align*}
For $r\in 2\cdot\N_0,k\in \N_0$, we define the globally existing vector bundle
\begin{align*}
    W^{k,r}:&=  \pi_Z^\ast (\underline{\Lambda^{k} E^\ast})\otimes_\C L^{k+r}\otimes_\C\pi_Z^\ast (\underline{\phi^r})
    = \pi_Z^\ast (\underline{\Lambda^{k} E^\ast})\otimes_\C L^{k}\otimes_\C\mathcal{L}^r= W^{k,0}\otimes_\C\mathcal{L}^r. 
\end{align*}
Let $\pi_F\colon F\rightarrow M$ be a $B_2$-vector bundle. We then define $\beta_{F,r}:=\beta\otimes_\C \Id_{\mathcal{L}^r}\otimes_\C\Id_{\pi_Z^\ast F}$ on $W^{k,r}  \otimes_\C\pi_Z^\ast F$, so
\begin{align*}
    \beta_{F,r}\colon \pi_Z^\ast (\underline{\Lambda^{k} E^\ast})\otimes_\C L^{k}\otimes_\C\mathcal{L}^r\otimes_\C\pi_Z^\ast F\rightarrow \Lambda^{0,k} T^\ast Z \otimes_\C\mathcal{L}^r\otimes_\C\pi_Z^\ast F.
\end{align*}
For $[p,([h],\eta\otimes_\C ({s'}^{\otimes l}_{ 1,[h]}{s'}^{\otimes (k-l)}_{2, [h]}))]\otimes_\C s\otimes_\C f\in \pi_Z^\ast (\underline{\Lambda^{k} E^\ast})\otimes_\C L^{k}\otimes_\C\mathcal{L}^r\otimes_\C\pi_Z^\ast F$ with $h=z_1h_1+z_2h_2$, this means that
\begin{align*}
    &\beta_{F,r} ([p,([h],\eta\otimes_\C ({s'}^{\otimes l}_{1, [h]}{s'}^{\otimes (k-l)}_{2, [h]}))]\otimes_\C s\otimes_\C f)\\
    =&(T\pi_Z)^\ast ([p,([h],\eta\otimes_\C \frac{z_1^{l}z_2^{k-l}}{\|h\|^{2k}}(J_r(h))^{\otimes k})])\otimes_\C s\otimes_\C f.
\end{align*}
This map now gives us the canonical map
\begin{align*}
    \beta_{F,r}\colon \Gamma (Z, W^{k,r}\otimes_\C\pi_Z^\ast F )\rightarrow \mathfrak{A}^{0,k} (Z,\pi_Z^\ast F\otimes_\C\mathcal{L}^r),
\end{align*}
which we also denote by $\beta_{F,r}$, and it is injective.
\subsection*{Construction of $\proj^V$ and $\dolb^V$}
For every $k\in \N_0$, $k>0$, $\Lambda^{0,k} T^{\ast}Z = (\Lambda^{0,k-1} T^{\ast H}Z\otimes_\C T^{\ast 0,1,V} Z)\oplus (\Lambda^{0,k} T^{\ast H}Z)$, and we denote by
\begin{align*}
\proj^V\colon \Lambda^{0,k} T^{\ast}Z \rightarrow \Lambda^{0,k-1} T^{\ast H}Z\otimes_\C T^{\ast 0,1,V} Z
\end{align*}
the projection onto the vertical component, since the decomposition $TZ=T^VZ\oplus T^HZ$ is compatible with the complex structure.\\
For $k>0$, we define
\begin{align*}
    \mathfrak{A}^{0,k} (Z, \mathcal{L}^r\otimes_\C\pi_Z^\ast F)^V&:= \Gamma (Z, \Lambda^{0,k-1} T^{\ast H}Z\otimes_\C T^{\ast 0,1,V} Z\otimes_\C \mathcal{L}^r\otimes_\C\pi_Z^\ast F),\\
    \mathfrak{A}^{0,k} (Z)^V&:= \Gamma (Z,\Lambda^{0,k-1} T^{\ast H}Z\otimes_\C T^{\ast 0,1,V} Z).
\end{align*}
This yields a maps $\proj^V\colon \mathfrak{A}^{0,k} (Z,\mathcal{L}^r\otimes_\C\pi_Z^\ast F) \rightarrow \mathfrak{A}^{0,k} (Z,\mathcal{L}^r\otimes_\C\pi_Z^\ast F)^V$ and $\proj^V\colon \mathfrak{A}^{0,k} (Z) \rightarrow \mathfrak{A}^{0,k} (Z)^V$, which we also denote by $\proj^V$.\\
For $w\in \Gamma (Z, W^{k,r}\otimes_\C\pi_Z^\ast F)$, we have $w_{\vert Z_y}\in \Gamma (Z_y, (\pi_Z^\ast (\underline{\Lambda^{k} E^\ast}) \otimes_\C L^{k+r}\otimes_\C\pi_Z^\ast (\underline{\phi^r})\otimes_\C\pi_Z^\ast F)_{\vert Z_y})$ for every $y\in M$ and locally
\begin{align*}
    w_{[p,[h]]}=\sum\limits_{(l_1,l_2,l_3,l_4)} q_{l_1}\cdot [p,([h],\eta_{l_2}\otimes_\C ({s'}_{1,[h]}^{\otimes l_3} {s'}_{2,[h]}^{\otimes (k+r-l_3)})\otimes_\C 1)]\otimes_\C \pi_Z^\ast f_{l_4}
\end{align*}
for $q_{l_1}\in C^\infty (\pi_Z^{-1} (U),\C),\eta_{l_2}\in \Lambda^k E^\ast,f_{l_4}\in \Gamma (U,F)$.\\
Now all vector bundle factors in $(\pi_Z^\ast (\underline{\Lambda^{k} E^\ast})\otimes_\C L^{k+r}\otimes_\C\pi_Z^\ast (\underline{\phi^r})\otimes_\C\pi_Z^\ast F)_{\vert Z_y}$ are trivial on $Z_y$ except for $L^{k+r}_{\vert Z_y}$. On the fibre $Z_y$, $L^{k+r}_{\vert Z_y}$ is a holomorphic line bundle by definition of the complex structure, but not on the whole of $Z$, and it corresponds to the holomorphic line bundle $\mathcal{O} (k+r)$. For this reason, we can define an operator
\begin{align*}
    \dolb^V\colon \Gamma (Z, W^{k,r}\otimes_\C\pi_Z^\ast F)\rightarrow \mathfrak{A}^{0,k+1} (Z, \mathcal{L}^r\otimes_\C\pi_Z^\ast F)^V
\end{align*}
fibrewise based on the Dolbeault operator of the fibre $Z_y$. Thus, the following applies for the section $w$ above
\begin{align*}
    &\dolb^V (w)_{ [p,[h]]} \\
    =& \sum\limits_{(l_1,l_2,l_3,l_4)} (\proj^V_{\vert \mathfrak{A}^{0,1} (Z)}\circ \dolb) (q_{l_1})\otimes_\C[p,([h],\eta_{l_2}\otimes_\C ({s'}_{1,[h]}^{\otimes l_3} {s'}_{2,[h]}^{\otimes (k+r-l_3)})\otimes_\C 1)]\otimes_\C\pi_Z^\ast f_{l_4}.
\end{align*}
Note that we regard $[p,([h],\eta_{l_2}\otimes_\C ({s'}_{1,[h]}^{\otimes l_3} {s'}_{2,[h]}^{\otimes (k+r-l_3)})\otimes_\C 1)]\otimes_\C\pi_Z^\ast f_{l_4}\in  W^{k,r}\otimes_\C\pi_Z^\ast F$ as an element of $ \Lambda^{0,k} T^{\ast H} Z\otimes_\C \mathcal{L}^r\otimes_\C\pi_Z^\ast F$ because of the isomorphism $W^{k,r}=  \pi_Z^\ast (\underline{\Lambda^ {k} E^\ast})\otimes_\C L^{k}\otimes_\C\mathcal{L}^r\cong \Lambda^{0,k} T^\ast Z^H\otimes_\C\mathcal{L}^r$.
\begin{remark}
    The map $\dolb^V\colon \Gamma (Z, W^{k,r}\otimes_\C\pi_Z^\ast F)\rightarrow \mathfrak{A}^{0,k+1} (Z, \mathcal{L}^r\otimes_\C\pi_Z^\ast F)^V$ is surjective and $\kernel(\dolb^V)$ consists of all sections that are holomorphic on every fibre.
\end{remark}
\noindent We have $0=H^{0,1}_{\dolb} (\mathbb{P}H,\mathcal{O}(k+r))=\kernel (\dolb_{\vert \mathfrak{A}^{0,1} (\mathbb{P}H,\mathcal{O} (k+r))}) / \image(\dolb_{\vert \mathfrak{A}^{0,0} (\mathbb{P}H,\mathcal{O} (k+r))})$. Therefore, we obtain $\image(\dolb_{\vert \mathfrak{A}^{0,0} (\mathbb{P}H,\mathcal{O} (k+r))}) = \kernel (\dolb_{\vert \mathfrak{A}^{0,1} (\mathbb{P}H,\mathcal{O} (k+r))}) = \mathfrak{A}^{0,1} (\mathbb{P}H,\mathcal{O} (k+r))$ because of $\dim_\C \mathbb{P}H=1$. This implies the claim of the remark.
\subsection*{Construction of $\mu_{F,r}$}
For $0\leq k\leq 2n$, there is a canonical map
\begin{align*}
    \mu_{F,r}\colon \Gamma (M,V^{k,r}\otimes_\C F)\rightarrow \Gamma (Z, W^{k,r}\otimes_\C \pi_Z^\ast F),
\end{align*}
which we define as follows:
Let $v\in \Gamma (M,V^{k,r}\otimes_\C F)$. Locally on $U\subset M$, $v$ is given by
\begin{align*}
    v=\sum\limits_{(l_1,l_2,l_3)}  [p,\eta_{l_1}\otimes_\C \frac{1}{(k+r)!} h_2^{l_2}\cdot (-h_1)^{k+r-l_2}\otimes_\C 1]\otimes_\C f_{l_3}
\end{align*}
for $p\in\Gamma (U,P),\eta_{l_1}\in\Lambda^k E^\ast,f_{l_3}\in \Gamma (U,F)$ and on $\pi_Z^{-1} (U)$ we define
\begin{align*}
    \mu_{F,r} (v)_{\vert z}:= \sum\limits_{(l_1,l_2,l_3)} [p_{ \pi_Z (z)},([h],\eta_{l_1}\otimes_\C{s'}_{1, [h]}^{\otimes l_2} {s'}_{2, [h]}^{\otimes (k+r-l_2)}\otimes_\C 1)] \otimes_\C {f}_{{l_3}, \pi_Z (z)}
\end{align*}
for $z=[p_x,[h]]\in \pi_Z^{-1} (U)$. It is obvious that this definition is well-defined and does not depend on the choice of neighbourhood, and that the map $\mu_{F,r}$ is injective.\\
\hfill\break
The commutativity of the diagram in Figure~\ref{quatlabel10} follows from the following results.
\begin{lemma}
    We have $\dolb^V=\proj^V\circ \dolb\circ \beta_{F,r}$.
\end{lemma}
\begin{proof}
    In \cite[Lemma~6.1]{salamonPhD}, Salamon has already shown this for the case where $r=0$ and $F$ is the trivial complex line bundle. Thus, the following diagram commutes
    \begin{center}
        \begin{tikzcd}
            \Gamma (Z,W^{k,0})\arrow{r}{\beta}\arrow[swap, near start]{drr}{\dolb^V} & \mathfrak{A}^{0,k}(Z) \arrow[near start,crossing over]{d}{\dolb} & \\
            &\mathfrak{A}^{0,k+1}(Z)\arrow[swap]{r}{\proj^V}& \mathfrak{A}^{0,k+1}(Z)^V.
        \end{tikzcd}
    \end{center}
    Let $f,s$ be local holomorphic sections in $\pi_Z^\ast F$ and $\mathcal{L}^r$, respectively, and let $w$ be a local section in $W^{k,0}$ on a neighbourhood of $z\in Z$. Due to the definition of the complex structure, the sections $f,s$ are also holomorphic when restricted to a fibre.\\
    This implies
    \begin{align*}
        &\proj^V\circ \dolb\circ \beta_{F,r} (w\otimes_\C s\otimes_\C f) 
        = \proj^V\circ \dolb\circ (\beta\otimes_\C\Id_{\mathcal{L}^r}\otimes_\C\Id_{\pi_Z^\ast F}) (w\otimes_\C s\otimes_\C f)\\
        =&\proj^V\circ \dolb (\beta (w)\otimes_\C s\otimes_\C f) = (\underbrace{\proj^V\circ \dolb\circ \beta}_{=\dolb^V} (w))\otimes_\C s\otimes_\C f
        = \dolb^V (w\otimes_\C s\otimes_\C f),
    \end{align*}
    whereas the last two transformations follow from the fibrewise holomorphy of $f$ and $s$.
\end{proof}
\begin{lemma}[{\cite[Lemma~6.2]{salamonPhD}}]
    If $F$ is the trivial complex line bundle and $r=0$, then $\beta\circ\mu \colon \Gamma (M,\underline{\Lambda^k E^\ast\otimes_\C\Sym^k H}) \rightarrow \mathfrak{A}^{0,k}(Z)$ is precisely the pullback of complex differential forms with $\pi_Z$ and a subsequent projection onto the $(0,k)$-forms because of $\underline{\Lambda^k E^\ast\otimes_\C\Sym^k H}\subset \Lambda^k (T^\ast M\otimes_\R\C)$.
\end{lemma}
\begin{lemma}
\label{quatlabel11}
    We have $\dolb\circ  (\beta_{F,r}\circ \mu_{F,r})=(\beta_{F,r}\circ \mu_{F,r})\circ\delta_{F,r}$ on $\Gamma (M,F\otimes_\C V^{k,r})$.\\
    Since $\beta_{F,r}\circ \mu_{F,r}$ is injective, the differential operator 
    \begin{align*}
        \dolb\colon \mathfrak{A}^{0,k}(Z,\pi_Z^\ast F\otimes \mathcal{L}^r)\rightarrow \mathfrak{A}^{0,k+1}(Z,\pi_Z^\ast F\otimes \mathcal{L}^r)
    \end{align*}
    induces the differential operator $\delta_{F,r}\colon \Gamma (M,F\otimes_\C V^{k,r})\rightarrow \Gamma (M,F\otimes_\C V^{k+1,r})$, i.e.\ we obtain $\delta_{F,r}=(\beta_{F,r}\circ \mu_{F,r})^{-1}\circ\dolb\circ (\beta_{F,r}\circ \mu_{F,r})$.
\end{lemma}
\begin{proof}
    In \cite[Lemma~6.3]{salamonPhD}, this was only proven for the case $r=0$ and the trivial complex line bundle $F$. We now want to prove this for the general case as well.\\
    Let $v\in\Gamma (M,F\otimes_\C V^{k,r})$. We must show that for every $z\in Z$, $\dolb\circ  (\beta_{F,r}\circ \mu_{F,r}(v))_{\vert z}=((\beta_{F,r}\circ \mu_{F,r})\circ\delta_{F,r}(v))_{\vert z}$ holds.\\
    Let $z\in Z$ and $x:=\pi_Z (z)$. We choose a local section $p\colon U\rightarrow P$ such that for the connection form $\omega\in\mathfrak{A}^1(P,\frg)$, $(p^\ast \omega)_{\vert x}=0$ holds as in Lemma~\ref{quatlabel31}. Furthermore, let $z=[p_x,[h_2]]$. Otherwise, there exists $\lambda\in\Sp (1)$ such that $\lambda (h_2)=h$ for every $h\in H$ with $\Vert h\Vert=1$. This yields $[p_x,[h]]=[p_x\circ \underbrace{[\Id,\lambda]}_{\in G},[h_2]]$, so in this case, replace $p\colon U\rightarrow P$ with $p\circ [\Id,\lambda] \colon U\rightarrow P$. Then $((p\circ [\Id,\lambda])^\ast \omega)_{\vert x}=0$ also applies, since $[\Id,\lambda]$ is constant (see \cite[Theorem~3.3(1)]{baumeichfeld}).\\
    Now let, without loss of generality, $v= [p,\eta\otimes_\C h_1^{k+r}\otimes_\C 1]\otimes_\C f$ with $f\in\Gamma (U,F),\eta:=\eta_1\wedge\ldots\wedge\eta_k\in \Lambda^kE^\ast$. We will see in the course of the calculation why we can limit ourselves to such sections because sections with $h_2$ as a factor vanish at $z=[p_x,[h_2]]$ due to the definition of the map $\mu_{F,r}$.\\
    For $z'=[p_y,[h]]\in\pi_Z^{-1} (U)$ with $h=z_1h_1+z_2h_2$, we have
    \begin{align*}
        &\beta_{F,r}\circ \mu_{F,r} ([p,\eta\otimes_\C h_1^{k+r}\otimes_\C 1 ]\otimes_\C 
        f)_{ z'}\\
        =& \beta_{F,r}((k+r)![p_{\pi_Z (z')},([h],\eta\otimes_\C{(-s')}_{2, [h]}^{\otimes (k+r)}\otimes_\C 1)]\otimes_\C f\circ \pi_Z (z'))\\
        =&(k+r)!(T_{z'}\pi_Z)^\ast ([p_{\pi_Z ({z'})},([h], \eta\otimes_\C\frac{z_2^k}{\|h\|^{2k}}(-J_r(h))^{\otimes k})]) \\
        &\otimes_\C [p_{\pi_Z (z')},([h],1\otimes_\C{(-s_{2, [h]}')}^{\otimes r})]\otimes_\C f\circ \pi_Z (z')\\
        =&\frac{(k+r)!}{k!}(T_{z'}\pi_Z)^\ast (\underbrace{[p_{\pi_Z ({z'})},([h], \eta\otimes_\C\frac{z_2^k}{\|h\|^{2k}}(-J_r(h))^{k})]}_{=:\alpha_{ z'}}) \otimes_\C \underline{(-s_2')}^{\otimes r}_{ z'}\otimes_\C f\circ \pi_Z (z'),
    \end{align*}
    and it follows
    \begin{align*}
        \dolb(\beta_{F,r}(\mu_{F,r} (v)))_{ z'}=& \frac{(k+r)!}{k!}\dolb((T\pi_Z)^\ast \alpha)_{ z'}\otimes_\C \underline{(-s_2')}^{\otimes r}_{ z'}\otimes_\C f\circ \pi_Z (z')
        \\
        &+ (-1)^k\frac{(k+r)!}{k!}((T_{z'}\pi_Z)^\ast \alpha)\wedge  (\underline{(-s_2')}^{\otimes r}_{ z'}\otimes \dolb (f\circ\pi_Z)_{ z'}).
    \end{align*}
    We have $\dolb((T\pi_Z)^\ast \alpha)_{\vert z}=\proj_{0,k+1}\circ (T_z\pi_Z)^\ast (d([p,\eta\otimes_\C h_1^{k}]))_{\vert z}=0$ because of \cite[Equation~(16)]{salamonPhD}. 
    According to \cite[Theorem~3.3.4]{kkdiffgeoENG}, the de Rham operator on $M$ factors through the torsion-free connection $\nabla$ on $\Lambda^k(T^\ast M\otimes_\R\C)$, which is induced by $\omega$, so it immediately follows that $(\nabla [p,\eta,h_1^{\otimes k}])_{\vert z}=0$. We have $\dolb^{\pi_Z^\ast F} (f\circ \pi_Z) = \proj_{0,1}\circ \nabla^{\pi_Z^\ast F} (f\circ \pi_Z)=\proj_{0,1}\circ \pi_Z^\ast (\nabla^Ff)$, so we denote $\nabla^Ff=\sum_{j} f_j\otimes [p,\eta_{j,1}\otimes_\C h_1+\eta_{j,2}\otimes_\C h_2]$. In summary, we obtain
    \begin{align*}
        &\dolb(\beta_{F,r}(\mu_{F,r} (v)))_{\vert z}
        =(-1)^k\frac{(k+r)!}{k!}((T_z\pi_Z)^\ast \alpha)\wedge  (\underline{(-s_2')}^{\otimes r}_{\vert z}\otimes_\C \dolb (f\circ\pi_Z)_{\vert z})\\
        =&(-1)^k\frac{(k+r)!}{k!}((T_z\pi_Z)^\ast \alpha)\wedge  (\underline{(-s_2')}^{\otimes r}_{\vert z}\\
        &\otimes_\C (\sum\limits_{j}(T_z\pi_Z)^\ast([p_{\pi_Z(z)},([h_2],\eta_{j,1}\otimes h_1)])f_j\circ\pi_Z(z)))\\
        =& \sum\limits_{j} \frac{(k+r)!}{k!}(T_z\pi_Z)^\ast ([p_{\pi_Z(z)},([h_2],\eta_{j,1}\otimes h_1)]\wedge \alpha)\otimes_\C \underline{(-s_2')}^{\otimes r}_{\vert z}\otimes_\C f_j\circ\pi_Z(z)\\
        =& \begin{aligned}[t]
            \sum\limits_{j} &\frac{(k+r)!}{k!}\frac{(k+1)!}{(k+1)}(T_z\pi_Z)^\ast ([p_{\pi_Z(z)},([h_2],(\eta_{j,1}\wedge\eta)\otimes h_1^{\otimes (k+1)})])\\
            &\otimes_\C \underline{(-s_2')}^{\otimes r}_{\vert z}\otimes_\C f_j\circ\pi_Z(z)
        \end{aligned}\\
        =&\sum\limits_{j} (k+r)!(T_z\pi_Z)^\ast ([p_{\pi_Z(z)},([h_2],(\eta_{j,1}\wedge\eta)\otimes h_1^{\otimes (k+1)})])\otimes_\C \underline{(-s_2')}^{\otimes r}_{\vert z}\otimes_\C f_j\circ\pi_Z(z)
    \end{align*}
    due to $\Lambda^{0,1}T^{\ast H} Z =(T_z\pi_Z)^\ast ( \lbrace [p_x,\eta\otimes J_r(h_2)]\mid \eta\in E^\ast\rbrace )$ and $h_1=-J_r(h_2)$.\\
    Therefore, we obtain
    \begin{align*}
        \delta_{F,r}(v)_{\vert x} =& \underline{\kappa}^F(\sum\limits_j [p_x,\eta_{j,1}\otimes_\C h_1+\eta_{j,2}\otimes_\C h_2]\otimes_\C [p_x,\eta\otimes_\C h_1^{k+r}\otimes_\C 1]\otimes {f}_{j, x})\\
        =& \sum\limits_j \frac{1}{k+r+1}[p_x,(\eta_{j,1}\wedge\eta)\otimes_\C h_1^{k+r+1}]\otimes {f}_{j, x} + (\text{terms with } h_2\text{ as a factor}).
    \end{align*}
    This yields
    \begin{align*}
        &\beta_{F,r}(\mu_{F,r}(\delta_{F,r}(v)))_{\vert z}\\
        =& \beta_{F,r}(\sum\limits_j \frac{(k+r+1)!}{k+r+1} [p_{ \pi_Z (z)},([h_2],((\eta_{j,1}\wedge\eta)\otimes_\C (-s_2')^{\otimes (k+r+1)}_{ [h_2]})]\otimes_\C {f_j}\circ \pi_Z (z)\\
        =& \sum\limits_j (k+r)!(T_z\pi_Z)^\ast ([p_{\pi_Z(z)},([h_2],(\eta_{j,1}\wedge\eta)\otimes h_1^{\otimes (k+1)})])\otimes_\C \underline{(-s_2')}^{\otimes r}_{\vert z}\otimes_\C f_j\circ\pi_Z(z),
    \end{align*}
    which proves the assertion.
\end{proof}
\begin{theorem}[{\cite[Theorem~6.2]{salamonPhD}}]
\label{quatlabel12}
    The map
    \begin{align*}
        \beta_{F,r}\circ\mu_{F,r} \colon \Gamma (M,V^{k,r}\otimes_\C F) \rightarrow \mathfrak{A}^{0,k}(Z,\mathcal{L}^r\otimes_\C\pi_Z^\ast F)
    \end{align*}
    induces an isomorphism $\psi \colon H^{k}_{\delta_{F,r}}(M)\rightarrow H^{0,k}_{\dolb}(Z,\mathcal{L}^r\tenc \pi_Z^\ast F)$, and we have
    \begin{align*}
        H^{0,2n+1}_{\dolb}(Z,\mathcal{L}^r\tenc \pi_Z^\ast F)=0.
    \end{align*}
\end{theorem}
\noindent The proof of this theorem for the general case does not really differ from the case considered in \cite[Theorem~6.1]{salamonPhD}, so we omit the proof of the theorem above.
\section{Fixed Point Theorems}
\label{quatsect4}
Let $M$ be a compact manifold, $E_0,\ldots,E_m$ vector bundles over $M$, and 
\begin{align*}
0\rightarrow\Gamma (M,E_0)\xrightarrow{D_0}\Gamma (M,E_1)\xrightarrow{D_1}\ldots \xrightarrow{D_{m-1}}\Gamma (M,E_m)\rightarrow 0
\end{align*}
an elliptic complex with differential operators $D_0,\ldots D_{m-1}$. For a $C^\infty$-map $\gamma \colon M\rightarrow M$ and vector bundle homomorphisms $\psi_k \colon \gamma^\ast E_k\rightarrow E_k$ for $k\in \lbrace 0,\ldots,m\rbrace$, we now define
\begin{align}
    \label{quatlabel13}
    \tau_k \colon &\Gamma (M,E_k)\stackrel{\gamma^\ast}{\rightarrow}\Gamma (M,\gamma^\ast E_k)\stackrel{\psi_k\circ}{\rightarrow} \Gamma (M,E_k)\notag\\
    &\tau_k (s) (x):= \psi_k (s(\gamma(x))).
\end{align}
If $D_k\circ \tau_k=\tau_{k+1}\circ D_k$ holds for all $k\in\lbrace 0,\ldots,m-1\rbrace$, then we call $\tau_\bullet$ an \textbf{endomorphism of the complex}, which always also provides maps
\begin{align*}
    H^k(\tau)\colon H^k(M,E)\rightarrow H^k(M,E)
\end{align*}
on the cohomology $H^k(M,E)=\dfrac{\kernel(D_k)}{\image (D_{k-1})}$ of the complex. We denote the set of all fixed points of a map $\gamma\colon M\rightarrow M$ by $M^\gamma$.\\
The main result of this article is the application of the Atiyah-Bott fixed point theorem in quaternionic geometry.
\begin{theorem}[{\cite[Theorem~A]{atiyahbott1}}, Atiyah-Bott fixed point theorem]
\label{quatlabel14}
    Let $M$ be a compact manifold, $E_0,\ldots,E_m$ vector bundles over $M$, and
    \begin{align*}
        0\rightarrow\Gamma (M,E_0)\xrightarrow{D_0}\Gamma (M,E_1)\xrightarrow{D_1}\ldots \xrightarrow{D_{m-1}}\Gamma (M,E_m)\rightarrow 0
    \end{align*}
    an elliptic complex with differential operators $D_0,\ldots D_{m-1}$. For a $C^\infty$-map $\gamma \colon M\rightarrow M$ with $\det(1-T_x\gamma)\neq 0$ for all $x\in M^\gamma$ and vector bundle homomorphisms $\psi_k \colon \gamma^\ast E_k\rightarrow E_k$, which provide an endomorphism of the complex as in (\ref{quatlabel13}), we have
    \begin{align*}
        \Tr_\mathrm{s} (H^\cdot (\tau))&=\sum\limits_{k=0}^m (-1)^k \Tr (H^k(\tau))
        = \sum\limits_{x\in M^\gamma} \frac{\sum\limits_{k=0}^m (-1)^k \Tr (\psi_{k, x})}{\vert \det\nolimits_\mathbb{R} (1-T_x \gamma)\vert}=\sum\limits_{x\in M^\gamma} \frac{\Tr_\mathrm{s} (\psi_{\cdot, x})}{\vert \det\nolimits_\mathbb{R} (1-T_x \gamma)\vert}.
    \end{align*}
\end{theorem}
\subsection{Quaternionic Maps of Manifolds}
\label{quatsect4.1}
We must now find out which maps $\gamma\colon M\rightarrow M$ and $\varphi \colon Z\rightarrow Z$ are suitable for our purposes, in order to obtain an endomorphism of the Salamon complex.
\begin{definition}[{\cite[Definition~2.4]{twistorialmaps}}]
\label{quatlabel26}
    A $C^\infty$-map $\gamma \colon M\rightarrow M$ of a quaternionic manifold is called \textbf{quaternionic} (with respect to $\varphi$) if there exists a $C^\infty$-map $\varphi \colon Z\rightarrow Z$ with $\gamma \circ \pi_Z = \pi_Z \circ \varphi$ such that
    \begin{align*}
        \varphi(J_z) \circ T_{\pi_Z (z)}\gamma = T_{\pi_Z (z)}\gamma \circ J_z
    \end{align*}
    holds for all $z\in Z$.
\end{definition}
\noindent We can now apply the results from Section~\ref{quatsect2.1}. Therefore, we know that for $p_x\in P_x,q_{\gamma (x)}\in P_{\gamma (x)}$, there always exists $[T,\lambda]\in\hnnsp{n}$ with
\begin{align*}
    q_{\gamma (x)}^{-1}\circ T_x\gamma \circ p_x = (v\mapsto Tv\lambda^{-1})
\end{align*}
and $\varphi([p_x,[h]])=[q_{\gamma (x)},[\lambda (h)]]$.\\
For a map $\gamma\colon M\rightarrow M$, we set $M_0:=\lbrace x\in M\mid T_x\gamma =0\rbrace$.
\begin{theorem}[{\cite[Theorem~3.5]{twistorialmaps}}]
\label{quatlabel15}
    Let $\gamma\colon M\rightarrow M$ and $\varphi\colon Z\rightarrow Z$ be $C^\infty$-maps for a quaternionic manifold $M$, such that $\mathring{M_0}=\emptyset$ and $\gamma \circ \pi_Z = \pi_Z \circ \varphi$. Then the following two statements are equivalent:
    \begin{enumerate}
        \item $\gamma\colon M\rightarrow M$ is quaternionic with respect to $\varphi \colon Z\rightarrow Z$.
        \item $\varphi \colon Z\rightarrow Z$ is holomorphic.
    \end{enumerate}
\end{theorem}
\noindent Due to the construction of the complex structure on $Z$, the holomorphy of $\varphi$ always implies that $\gamma$ is quaternionic, even without the assumption $\mathring{M_0}=\emptyset$.
\begin{remark}
    For a Hyperk\"ahler manifold $M$, i.e.\ $M$ possesses a $\spn{n}$-structure, there exists a globally defined hypercomplex structure $(I,J,K)$ such that $I,J,K\in \Gamma (M,Z)$. In \cite{ChenLiQuaternionicMaps}, a $C^\infty$-map $\gamma\colon M\rightarrow M$ is called quaternionic if there exists an $A\in \SO(3)$ such that 
    \begin{align*}
        A(I)\circ T\gamma = T\gamma \circ I, && A(J)\circ T\gamma = T\gamma \circ J, && A(J)\circ T\gamma = T\gamma \circ J.
    \end{align*}
    Since there are no globally defined hypercomplex structures for general quaternionic manifolds, such as $\mathbf{P}^n\mathbb{H}$, this definition can not be directly extended to quaternionic manifolds. For Hyperk\"ahler manifolds, the main difference between the definitions of quaternionic maps lies in the fact that the definition from \cite{twistorialmaps} permits a variation of $A\in\SO(3)$, which expresses itself precisely in the map $\varphi\colon Z\rightarrow Z$, that is why every quaternionic map following \cite{ChenLiQuaternionicMaps} is, of course, also a quaternionic map following \cite{twistorialmaps}.\\
    For a Hyperk\"ahler manifold $M$, $Z$ can be trivialised with the help of the globally defined hypercomplex structure, i.e.\ we have $Z = M \times S^2$ and $\varphi(x,z) = (\gamma(x), A(z))$ for a constant $A \in \SO(3)$. However, following the definition from \cite{twistorialmaps}, $A\in\Gamma(M,\R^{3\times 3})$ would be possible, provided that $A_x(S^2)\subset S^2$ holds for all $x\in M$.
\end{remark}
\begin{lemma}
\label{quatlabel16}
    Let $\gamma\colon M\rightarrow M$ be a quaternionic map such that $\mathring{M_0}=\emptyset$.\\
    For $x\in M$ and local sections $p\colon U_x\rightarrow P$ and $q\colon U_{\gamma (x)}\rightarrow P$ for neighbourhoods of $x\in U_x$ and $\gamma (x)\in U_{\gamma (x)}$, respectively, there exist $T\in C^\infty (U'_x,\mathbb{H}^{n\times n})$ and $\lambda\in C^\infty (U'_x,\Sp (1))$ for a neighbourhood $x\in U_x'\subset U_x\cap \gamma^{-1} (U_{\gamma (x)})$, such that
    \begin{align*}
        T_y\gamma = q_{\gamma (y)}\circ (v\mapsto T_yv\lambda_y^{-1})\circ p_y^{-1}
    \end{align*}
    holds for all $y\in U_x'$.\\
    In particular, $\varphi ([p_y,[h]])=[q_{\gamma (y)},[\lambda_y (h)]]$ then holds for all $y\in U_x'$ for the $\varphi \colon Z\rightarrow Z$ with respect to which $\gamma$ is quaternionic.
\end{lemma}
\begin{proof}
    Without loss of generality, we assume $\gamma (U_x)\subset U_{\gamma (x)}$.\\
    Let
    \begin{align*}
        &I:=p\circ R_{-i}\circ p^{-1},J:=p\circ R_{-j}\circ p^{-1},K:=p\circ R_{-k}\circ p^{-1},\\
        &I':=q\circ R_{-i}\circ q^{-1},J':=q\circ R_{-j}\circ q^{-1},K':=q\circ R_{-k}\circ q^{-1}.
    \end{align*}
    Then $\varphi (I),\varphi (J),\varphi (K)$ and $I'\circ \gamma,J'\circ \gamma,K'\circ\gamma$ are sections of $\gamma^\ast Q_{\vert U_x}$.\\
    Since the latter triple is a basis of $\gamma^\ast Q_{\vert U_x}$, $\varphi(I),\varphi(J),\varphi(K)$ can each be written as a linear combination of this basis. We denote the coefficient matrix with respect to this basis by $A\in\Gamma (U_x,\R^{3\times 3})$.\\
    For $y\in M_0^c\cap U_x$, we have $A_y\in \SO (3)$, as seen in Section~\ref{quatsect2}. Since $\SO (3)\subset \R^{3\times 3}$ is a closed subset and $\mathring{M_0}=\emptyset$, it follows from the continuity of $A$ that $A_y\in \SO (3)$ for all $y\in U_x$.\\
    There exists a neighbourhood of $A_x\in\SO (3)$ such that the 2-fold covering $\Sp(1)\rightarrow \SO (3)$ is a diffeomorphism on this neighbourhood. This gives us a neighbourhood $U_x'\subset U_x$ of $x$ on which the section $A\in\Gamma (U_x',\SO (3))$ then gives us a section $\lambda\in\Gamma(U_x',\Sp (1))$.\\
    Now consider $B_y:=q_{\gamma (y)}^{-1}\circ T_y\gamma \circ p_y$ for $y\in U_x'$, then $B\in \Gamma (U_x',\R^{4n\times 4n})$ and it follows that $R_\lambda\circ B=: T\in \Gamma (U_x',\mathbb{H}^{n\times n})$, so $B(v)=Tv\lambda^{-1}$. This proves the desired statement.
\end{proof}
\noindent If we replace the sections $p$ and $q$ with other suitable sections, we obtain the following corollary.
\begin{corollary}
\label{quatlabel17}
Under the conditions of Lemma~\ref{quatlabel16}, one even obtains the following statements.
\begin{enumerate}
    \item There are local sections $p\colon U_x'\rightarrow P, q\colon U_{\gamma (x)}'\rightarrow P$, such that $T_y\gamma = q_{\gamma (y)}\circ (v\mapsto T_yv)\circ p_y^{-1}$ for all $y\in U_x'$, i.e.\ in particular $\varphi ([p_y,[h]])=[q_{\gamma (y)}, [h]]$ for all $y\in U_x'$.
    \item If $x\in M$ is a fixed point of $\gamma$, then there exists a neighbourhood $x\in U_x'$ and a local section $p\colon U_x'\rightarrow P$ such that
    \begin{align*}
        T_y\gamma = p_{\gamma (y)}\circ (v\mapsto T_yv\lambda_y^{-1})\circ p_y^{-1}
    \end{align*}
    for all $y\in U_x'\cap\gamma^{-1} (U_{\gamma (x)}')$ and $\lambda_x=\begin{pmatrix}
        e^{it_x}&0\\0&e^{-it_x}
    \end{pmatrix}$ for some $t_x\in\R$. \label{quatlabel18}
    \item If $T_y\gamma$ is invertible for all $y\in U_x$ on a neighbourhood $U_x$ of $x\in M$, then there exist local sections $p\colon U_x'\rightarrow P$, $q\colon U_{\gamma (x)}'\rightarrow P$ such that $T_y\gamma = q_{\gamma (y)}\circ p_y^{-1}$ for all $y\in U_x'$.
\end{enumerate}
\end{corollary}
For given $\gamma\colon M\rightarrow M$ and $\varphi\colon Z\rightarrow Z$, we now construct the maps for each tensor factor of the spaces, which altogether give us an endomorphism of the Dolbeault complex of $Z$ with values in $\mathcal{L}^r\tenc\pi_Z^\ast F$ or of the Salamon complex.\\
For this purpose, always take $r\in 2\cdot\mathbb{N}_0$.
\subsection*{Construction on the $\mathcal{L}^r$ factor}
Let $r>0$. Furthermore, let $\varphi\colon Z\rightarrow Z$ be holomorphic and $\gamma\colon M\rightarrow M$ be given with $\gamma\circ \pi_Z = \pi_Z \circ \varphi$, such that $T_x\gamma$ is invertible for all $x\in M$. The requirement of invertibility is necessary, as we shall see.\\
For $x\in M$ and $p_x\in P_x,q_{\gamma (x)}\in P_{\gamma (x)}$, there exists $[T,\lambda]\in\glsp{n}$ such that $q_{\gamma (x)}^{-1}\circ T_x\gamma \circ p_x = (v\mapsto Tv\lambda^{-1})$ and $\varphi([p_x,[h]])=[q_{\gamma (x)},[\lambda (h)]]$.\\
On $\varphi^\ast\mathcal{L}^r_{\vert \pi_Z^{-1} (x)}$, we define
\begin{align*}
    \varphi^{\mathcal{L}^r}([p_x,[h]],[q_{\gamma (x)},([\lambda (h)], s'\tenc 1)])
    :=& [p_x,[T,\lambda]^{-1}\cdot ([\lambda (h)],s'\tenc 1)]\\
    =& [p_x,([h],\det\nolimits_\C (T)^{-\frac{r}{2(n+1)}}(s'\circ \lambda^{\otimes r})\tenc 1)].
\end{align*}
This construction is well-defined:\\
Let $p'_x\in P_x,q'_{\gamma (x)}\in P_{\gamma (x)}$, then there exist $[T_1,\lambda_1],[T_2,\lambda_2]\in\glsp{n}$ such that $p'_x=p_x\circ [T_1,\lambda_1], q'_{\gamma (x)}=q_{\gamma (x)}\circ [T_2,\lambda_2]$ and we have
\begin{align*}
    &{q'_{\gamma (x)}}^{-1}\circ T_x\gamma \circ p'_x = [T_2^{-1},\lambda_2^{-1}]\circ q_{\gamma (x)}^{-1} \circ T_x\gamma \circ p_x\circ [T_1,\lambda_1]\\
    =& [T_2^{-1},\lambda_2^{-1}]\circ (v\mapsto Tv\lambda^{-1})\circ [T_1,\lambda_1]
    = (v\mapsto T_2^{-1}TT_1v(\lambda_2^{-1}\lambda\lambda_1)^{-1}).
\end{align*}
The construction is independent of the choice of $p_x$ and $q_{\gamma (x)}$ because of
\begin{align*}
    &\varphi^{\mathcal{L}^r}([p_x,[h]],[q_{\gamma (x)},([\lambda (h)], s'\tenc 1)])\\
    =& \varphi^{\mathcal{L}^r}([p_x',[T_1,\lambda_1]^{-1}\cdot[h]],[q'_{\gamma (x)},[T_2,\lambda_2]^{-1}\cdot ([\lambda (h)], s'\tenc 1)])\\
    =& \varphi^{\mathcal{L}^r}([p_x',[\lambda_1^{-1} (h)]],[q'_{\gamma (x)},([\lambda_2^{-1}\lambda (h)], \det\nolimits_\C (T_2)^{-\frac{r}{2(n+1)}} (s'\circ \lambda_2^{\otimes r})\tenc 1)])\\
    =& [p_x',([\lambda_1^{-1} (h)],\det\nolimits_\C (T_2)^{-\frac{r}{2(n+1)}} \det\nolimits_\C (T_2^{-1}TT_1)^{-\frac{r}{2(n+1)}}(s'\circ (\lambda_2\lambda_2^{-1}\lambda\lambda_1)^{\otimes r})\tenc 1)]\\
    =& [p_x',([\lambda_1^{-1} (h)],\det\nolimits_\C (T)^{-\frac{r}{2(n+1)}} \det\nolimits_\C (T_1^{-1})^{\frac{r}{2(n+1)}}(s'\circ (\lambda\lambda_1)^{\otimes r})\tenc 1)]\\
    =& [p'_x,[T_1,\lambda_1]^{-1}\cdot ([h],\det\nolimits_\C (T)^{-\frac{r}{2(n+1)}} (s'\circ\lambda^{\otimes r})\tenc 1)]\\
    =& [p_x, ([h],\det\nolimits_\C (T)^{-\frac{r}{2(n+1)}} (s'\circ\lambda^{\otimes r})\tenc 1)].
\end{align*}
Altogether, we obtain a map $\varphi^{\mathcal{L}^{r}}\colon \varphi^\ast\mathcal{L}^{r}\rightarrow \mathcal{L}^{r}$, which, due to Lemma~\ref{quatlabel16}, is indeed a $C^\infty$-map.\\
For the holomorphy of the map, we choose, as in Corollary~\ref{quatlabel17}, local sections $p\colon U_x'\rightarrow P$, $q\colon U_{\gamma (x)}'\rightarrow P$ for $x\in M$  with $T_y\gamma =q_{\gamma (y)}\circ p_y^{-1}$ for all $y\in U_x'$.\\
For an element $s'\in\Gamma^{\operatorname{hol}}(\mathbb{P}H,\mathcal{O} (r))$, we consider the local holomorphic section 
\begin{align*}
    s_{[p_y,[h]]}:=([p_y,[h]],[q_{\gamma (y)},([h],s'_{[h]}\otimes_\C 1)])
\end{align*}
of $\varphi^\ast\mathcal{L}^r_{\vert \pi_Z^{-1} (U_x)}$. We have
\begin{align*}
    \varphi^{\mathcal{L}^r}(s)_{[p_y,[h]]}=[p_y,([h],s'_{[h]}\otimes_\C 1)],
\end{align*}
which is also a local holomorphic section of $\mathcal{L}^r$. This makes it easy to see that the map $\varphi^{\mathcal{L}^{r}}$ is in fact holomorphic.\\
We denote by $\tau^{\mathcal{L}^r}$ the map from (\ref{quatlabel13}) with respect to $\varphi^{\mathcal{L}^r}$. We denote by $\tau^Z\colon \mathfrak{A}^{0,k}(Z) \rightarrow \mathfrak{A}^{0,k+1}(Z)$ the pullback of forms with $\varphi$. Due to holomorphy, it follows that $\tau^Z\tenc\tau^{\mathcal{L}^r}$ commutes with the Dolbeault operator.
\subsection*{Construction on $V^{k,r}$}
For a holomorphic map $\varphi\colon Z\rightarrow Z$, the pullback with $\varphi$ of $(0,\bullet)$-forms commutes with the Dolbeault operator. We would now like to use this property for a quaternionic map $\gamma\colon M\rightarrow M$, together with the construction on the $\mathcal{L}^r$-factor, to construct a $\gamma^{V^{k,r}}\colon \gamma^\ast V^{k,r}\rightarrow V^{k,r}$, which is compatible with the Salamon complex.\\
Let $\gamma\colon M\rightarrow M$ be a quaternionic map with respect to $\varphi\colon Z\rightarrow Z$, such that $\mathring{M_0}=\emptyset$. Then, according to Theorem~\ref{quatlabel15}, $\varphi$ is holomorphic. In the case of $r\neq 0$, we additionally require that $T_x\gamma$ is invertible for all $x\in M$. With $\tau^Z\colon \mathfrak{A}^{0,k}(Z)\rightarrow \mathfrak{A}^{0,k+1}(Z)$ as defined above, we thus define $\tau^{V^{k,r}}: =(\beta_r\circ\mu_{r})^{-1}\circ(\tau^Z\otimes \tau^{\mathcal{L}^r})\circ (\beta_r\circ\mu_{r})$ for $r\neq 0$ and $\tau^{V^{k,r}}: =(\beta\circ\mu)^{-1}\circ\tau^Z\circ (\beta\circ\mu)$ for $r=0$. Due to Lemma~\ref{quatlabel11}, compatibility follows from the holomorphy of $\varphi$ and $\varphi^{\mathcal{L}^r}$.\\
It is very easy to verify that $\tau^{V^{k,r}}$ is induced, just as in (\ref{quatlabel13}), by a map $\gamma^{V^{k,r}}$ given by 
\begin{align*}
    &\gamma^{V^{k,r}}((x,[q_{\gamma (x)},(\eta_{1}\wedge\ldots \wedge \eta_{k})\otimes_\C (h_{1}\cdot\ldots\cdot h_{k+r})\otimes c]))\\
    :=&[p_x,(\eta_{1}\circ T_x\wedge\ldots \wedge \eta_{k}\circ T_x)\otimes_\C (\lambda_x^{-1}(h_{1})\cdot\ldots\cdot \lambda_x^{-1}(h_{k+r}))\otimes (\det\nolimits_\C(T_x))^{-\frac{r}{2(n+1)}}c]
\end{align*}
for $T_x\gamma = q_{\gamma (x)}\circ (v\mapsto T_xv\lambda_x^{-1})\circ p_x^{-1}$. For $r=0$, $\tau^{V^{k,0}}$ is precisely the pullback of forms on $M$. For $r\neq 0$, it follows from Corollary~\ref{quatlabel17} that $\gamma^{V^{k,r}}$ is a $C^\infty$-map.\\
For $r=0$, this follows from the fact that $\gamma^{V^{k,0}}$ is the pullback of forms, thus in this case we do not even need $\mathring{M_0}=\emptyset$. We can even use a more general definition of quaternionic maps $\gamma\colon M\rightarrow M$ which only imposes a pointwise condition on $T\gamma$ instead of the definition given in Definition~\ref{quatlabel26}.
\begin{remark}
\label{quatlabel28}
   Our results still apply if we just require $T_x\gamma\colon T_xM\rightarrow T_{\gamma (x)}M$ to be a quaternionic map of vector spaces with respect to their quaternionic structures $Q_x$ and $Q_{\gamma (x)}$ in the sense of Definition~\ref{quatlabel27} for all $x\in M$. In this case, the compatibility can easily be proven with the decomposition given in Remark~\ref{quatlabel9} because the pullback of forms commutes with the de Rham operator, and it preserves the spaces of this decomposition. 
\end{remark}
\subsection*{Construction on $F$ and $\pi_Z^\ast F$}
Let $\pi_F\colon F\rightarrow M$ be a $B_2$-vector bundle with its connection $\nabla^F$. Furthermore, let $\gamma\colon M\rightarrow M$ be a quaternionic map with respect to $\varphi\colon Z\rightarrow Z$. We denote the pullback of forms with $\gamma$ by $\tau^M\colon \mathfrak{A}^k_\C(M)\rightarrow \mathfrak{A}^{k+1}_\C(M)$. For a vector bundle homomorphism $\gamma^F\colon \gamma^\ast F\rightarrow F$, we also write $\tau^F$ for the map from (\ref{quatlabel13}). Similarly, for a vector bundle homomorphism $\varphi^{\pi_Z^\ast F}\colon \varphi^\ast (\pi_Z^\ast F)\rightarrow \pi_Z^\ast F$, we denote the map from (\ref{quatlabel13}) by $\tau^{\pi_Z^\ast F}$.\\
Because of $\gamma \circ \pi_Z = \pi_Z \circ \varphi$, we obtain
\begin{align}
    \varphi^\ast (\pi_Z^\ast F) =& \lbrace (z,(\varphi (z),f))\mid z\in Z, (\varphi (z),f)\in (\pi_Z^\ast F)_{\varphi (z)}\rbrace\nonumber\\
    =& \lbrace (z,(\varphi (z),f)) \mid z\in Z, f\in \underbrace{F_{\pi_Z (\varphi(z))}}_{=F_{\gamma (\pi_Z(z))}} \rbrace\notag\\
    \cong & \lbrace (z,(\pi_Z (z),f)) \mid z\in Z, (\pi_Z (z), f)\in (\gamma^\ast F)_{\pi_Z (z)} \rbrace
    = \pi_Z^\ast (\gamma^\ast F). \label{quatlabel19}
\end{align}
\begin{lemma}
\label{quatlabel20}
    Let $\varphi \colon Z\rightarrow Z$ be holomorphic and $\gamma\colon M\rightarrow M$ with $\gamma\circ \pi_Z = \pi_Z \circ \varphi$, such that a map $\gamma^{V^{k,r}}$ compatible with $\delta_r$ exists as above.\\
    A holomorphic vector bundle homomorphism $\varphi^{\pi_Z^\ast F}\colon \varphi^\ast (\pi_Z^\ast F)\rightarrow \pi_Z^\ast F$ induces a vector bundle homomorphism $\gamma^F\colon \gamma^\ast F\rightarrow F$ that is compatible with $\delta_{F,r}$. In particular, $\varphi^{\pi_Z^\ast F}_{\vert Z_x}$ is constant for all $x\in M$.
\end{lemma}
\begin{proof}
    We first note that, due to the complex structure of $Z$ and $\pi_Z^\ast F$, we now have $\varphi^\ast (\pi_Z^\ast F)_{\vert Z_x} \cong Z_x\times  F_{\gamma (x)}$ and $\pi_Z^\ast F_{\vert Z_x} \cong Z_x\times  F_x$ are trivial holomorphic vector bundles on the complex submanifolds $Z_x\subset Z$ for $x\in M$.\\
    If we restrict $\varphi^{\pi_Z^\ast F}$ to the fibre of $Z_x$, then
    \begin{align*}
        \varphi^{\pi_Z^\ast F}_{\vert Z_x}\colon Z_x\times F_{\gamma (x)} \rightarrow Z_x\times F_x 
    \end{align*}
    and we can view $\varphi^{\pi_Z^\ast F}_{\vert Z_x}$ as a holomorphic section in $\Hom_\C(F_{\gamma (x)},F_x)$ which is constant due to the compactness of $Z_x$. So for every $x\in M$, there is a $\gamma^F (x,\cdot):=\gamma^F_x\in \Hom_\C(F_{\gamma (x)},F_x)$ with $\varphi^{\pi_Z^\ast F} (z,(\pi_Z (z),f))=(z,\gamma^F ({\pi_Z (z)},f))$. This gives us a $C^\infty$-map $\gamma^F\colon \gamma^\ast F\rightarrow F$. Compatibility with $\delta_{F,r}$ follows again from Lemma~\ref{quatlabel11}.
\end{proof}
\begin{lemma}
    Let $\gamma \colon M\rightarrow M$ and $\gamma^{V^{\cdot,r}}$ be compatible with $\delta_r$ as above. If a vector bundle homomorphism $\gamma^F\colon \gamma^\ast F\rightarrow F$ now meets
    \begin{align*}
        \nabla^F \circ \tau^F = (\tau^M\tenc \tau^F) \circ \nabla^F
    \end{align*}
    then $\gamma^{V^{\cdot,r }}\tenc \gamma^F \colon \gamma^\ast (V^{\cdot ,r}\tenc F)\rightarrow V^{\cdot ,r}\tenc F$ is compatible with $\delta_{F,r}$.\\
    If, in addition, a holomorphic $\varphi\colon Z\rightarrow Z$ with $\gamma \circ \pi_Z = \pi_Z \circ \varphi$ exists, then $\gamma^F$ induces a holomorphic $\varphi^{\pi_Z^\ast F}\colon \varphi^\ast (\pi_Z^\ast F)\rightarrow \pi_Z^\ast F$ using the identification (\ref{quatlabel19}) by defining
    \begin{align*}
        \varphi^{\pi_Z^\ast F} \colon \varphi^\ast (\pi_Z^\ast F)&\rightarrow \pi_Z^\ast F\\
        (z,(\pi_Z (z),f))&\mapsto (z,\gamma^F(\pi_Z (z),f)).
    \end{align*}
\end{lemma}
\begin{proof}
\begin{enumerate}
    \item Let $s\tenc f\in \Gamma (M,V^{k,r}\tenc F)$. We have
    \begin{align*}
        &\underline{\kappa}^F (\gamma^{V^{k,r}}(s\circ \gamma)\tenc \nabla^F \gamma^F(f\circ \gamma))= \underline{\kappa}^F (\tau^{V^{k,r}}(s)\tenc \nabla^F \tau^F(f))\\
        =& \underline{\kappa}^F ((\tau^M\tenc \tau^{V^{k,r}}\tenc \tau^F)(s\tenc \nabla^F f))= (\tau^{V^{k+1,r}}\tenc \tau^F)\circ \underline{\kappa}^F (s\tenc \nabla^F f).
    \end{align*}
    This implies
    \begin{align*}
        &\delta_{F,r} \circ (\tau^{V^{k+1,r}}\tenc \tau^F) (s\tenc f) = \delta_{F,r} ((\gamma^{V^{k,r}}\tenc \gamma^F)(s\circ \gamma\tenc f\circ \gamma))\\
        = & (\delta_r\tenc 1 + \underline{\kappa}^F\circ (1\tenc \nabla^F)) (\gamma^{V^{k,r}}(s\circ \gamma)\tenc \gamma^F(f\circ \gamma))\\
        = & \delta_r(\gamma^{V^{k,r}}(s\circ \gamma))\tenc \gamma^F(f\circ \gamma) + \underline{\kappa}^F (\gamma^{V^{k,r}}(s\circ \gamma)\tenc \nabla^F \gamma^F(f\circ \gamma))\\
        = & (\gamma^{V^{k+1,r}} \circ \gamma^\ast (\delta_r(s)))\tenc \gamma^F(f\circ \gamma) + (\tau^{V^{k+1,r}}\tenc \tau^F)\circ \underline{\kappa}^F (s\tenc \nabla^F f)\\
        =& (\tau^{V^{k+1,r}}\tenc \tau^F) (\delta_r(s)\tenc f+ \underline{\kappa}^F (s\tenc \nabla^F f))
        = (\tau^{V^{k+1,r}}\tenc \tau^F) \circ \delta_{F,r}(s\tenc f).
    \end{align*}
    \item The vector bundle homomorphism $\varphi^{\pi_Z^\ast F}$ is holomorphic. Similar to what was done above, the construction of the complex structure on $\pi_Z^\ast F$ implies that $\dolb \circ \tau^{\pi_Z^\ast F} = (\tau^Z\tenc \tau^{\pi_Z^\ast F})\circ \dolb$ holds. This proves the desired statement.
\end{enumerate}
\end{proof}
\subsection{Fixed Point Theorem for Quaternionic Manifolds}
\label{quatsect4.2}
\begin{remark}
    For $r=0$, as explained in Remark~\ref{quatlabel28}, we can apply the more general definition of quaternionic maps in Theorem~\ref{quatlabel33}.
\end{remark}
\begin{proof}[Proof of Theorem~\ref{quatlabel33}]
The conditions of the Atiyah-Bott fixed point theorem~\ref{quatlabel14} are already given here, so we will calculate
\begin{align*}
    \frac{\Tr_\mathrm{s} ((\gamma^{V^{\bullet,r}}\tenc \gamma^F)_{\vert x})}{\vert \det\nolimits_\mathbb{R} (1-T_x \gamma)\vert}
\end{align*}
for all $x\in M^\gamma$.\\
Let $x\in M^\gamma$. Because of $TM\tenr \C\cong \underline{E\tenc H}$, we have 
\begin{align*}
    \vert \det\nolimits_\mathbb{R} (1-T_x \gamma)\vert = \det\nolimits_\C(1-e^{it_x}T_x) \det\nolimits_\C(1-e^{-it_x}T_x).
\end{align*}
Let $\lambda_x:=\begin{pmatrix}
    e^{it_x}&0\\0&e^{-it_x}
\end{pmatrix}$. We need to calculate
\begin{align*}
    \Tr_\mathrm{s} ((\gamma^{V^{\bullet,r}}\tenc \gamma^F)_{\vert x}) =& \sum\limits_{k=0}^{2n} (-1)^k \Tr ((\gamma^{V^{\bullet,r}}\tenc \gamma^F)_{\vert x})\\
    =&\sum\limits_{k=0}^{2n} (-1)^k \det\nolimits_\C (T_x)^{-\frac{r}{2(n+1)}} \Tr \left({T_x^\ast}^{\otimes k}_{\vert \Lambda^k E^\ast}\right) \Tr \left({\lambda_x^{-1}}^{\otimes k+r}_{\vert \symh{k+r}}\right) \Tr (\gamma^F_{\vert x}),
\end{align*}
so we distinguish between the two cases mentioned in the theorem:
\begin{enumerate}
    \item $e^{it_x}\not\in \lbrace -1,1\rbrace$, i.e.\ $t_x\in \R/\pi\cdot\Z$: As in the proof of \cite[Lemma~3.2]{kkquattorsion}, we have 
    \begin{align*}
        \Tr \left({\lambda_x^{-1}}^{\otimes k+r}_{\vert \symh{k+r}}\right)= \frac{e^{i(k+r+1)t_x}-e^{-i(k+r+1)t_x}}{e^{it_x}-e^{-it_x}}
    \end{align*}
    and
    \begin{align*}
        \sum\limits_{k=0}^{2n} (-1)^k\Tr \left({e^{\pm it_x}T_x^\ast}^{\otimes k}_{\vert \Lambda^k E^\ast}\right) = \det\nolimits_\C(1-e^{\pm it_x}T_x)
    \end{align*}
    In summary, this will give us
    \begin{align*}
        &\sum\limits_{k=0}^{2n} (-1)^k \det\nolimits_\C (T_x)^{-\frac{r}{2(n+1)}} \Tr \left({T_x^\ast}^{\otimes k}_{\vert \Lambda^k E^\ast}\right) \Tr \left({\lambda_x^{-1}}^{\otimes k+r}_{\vert \symh{k+r}}\right) \Tr (\gamma^F_{\vert x})\\
        =&  \frac{\Tr (\gamma^F_{\vert x}) \det\nolimits_\C (T_x)^{-\frac{r}{2(n+1)}}}{e^{it_x}-e^{-it_x}}\left(e^{i(r+1)t_x}\det\nolimits_\C(1-e^{it_x}T_x)-e^{-i(r+1)t_x}\det\nolimits_\C(1-e^{-it_x}T_x)\right)\\
        =& \Tr (\gamma^F_{\vert x}) \det\nolimits_\C (T_x)^{-\frac{r}{2(n+1)}}\left(\frac{e^{irt_x}\det\nolimits_\C(1-e^{it_x}T_x)}{1-e^{-2it_x}} + \frac{e^{-irt_x}\det\nolimits_\C(1-e^{-it_x}T_x)}{1-e^{2it_x}}\right).
    \end{align*}
    Thus, we get the desired statement for $\omega(x)$.
    \item $e^{it_x}\in \lbrace -1,1\rbrace$, i.e.\ $t_x\in\pi\cdot\Z$: Let $e^{it_x}=1$, otherwise replace $T_x$ with $-T_x$.\\
    We have $\Tr \left({\lambda_x^{-1}}^{\otimes k+r}_{\vert \symh{k+r}}\right)=\dim_\C(\symh{k+r}) = k+r+1$. With $(-1)^k(k+r+1)=-(-1)^{k+r+1}(k+r+1)$ and \cite[Equation~(27)]{kkhermsymm}, we obtain
    \begin{align*}
        &\sum\limits_{k=0}^{2n} (-1)^k(k+r+1) \Tr \left({T_x^\ast}^{\otimes k}_{\vert \Lambda^k E^\ast}\right)= \frac{\partial}{\partial s}_{\vert s=1} \sum\limits_{k=0}^{2n} -(-s)^{k+r+1} \Tr \left({T_x^\ast}^{\otimes k}_{\vert \Lambda^k E^\ast}\right) \\
        =& \frac{\partial}{\partial s}_{\vert s=1} \det\nolimits_\C(1-sT_x)s^{r+1} = \det\nolimits_\C(1-T_x)(r+1-\Tr ((1-T_x)^{-1}T_x)).
    \end{align*}
    Since $\vert \det\nolimits_\mathbb{R} (1-T_x \gamma)\vert =\det\nolimits_\C(1-T_x)^2$, the desired statement for $\omega(x)$ follows.
\end{enumerate}
\end{proof}
\noindent In the fixed point theorem for $Z$, we will see that the two summands of each fixed point of $\gamma$, which occur in the case $e^{it_x}\not\in \lbrace -1,1\rbrace$ for all $x\in M^\gamma$, each belong to a fixed point on $Z$.
\subsection{Fixed Point Theorem for Twistor Spaces}
\label{quatsect4.3}
\begin{lemma}
\label{quatlabel22}
    Let $\varphi\colon Z\rightarrow Z$ be holomorphic and $\gamma \colon M\rightarrow M$ quaternionic such that $\gamma \circ \pi_Z = \pi_Z \circ \varphi$, so that $T_x\gamma \neq 0$ for all fixed points $x\in M$ of $\gamma$ and $\det_\R(1-T_z\varphi)\neq 0$ for all fixed points $z\in Z$ of $\varphi$. Then the fixed points of $\varphi$ occur in pairs in $\pi_Z^{-1}(M^\gamma)$.
\end{lemma}
\begin{proof}
We can easily see $\pi_Z(Z^\varphi)\subset M^\gamma$. We have $e^{it_x}\not\in \lbrace -1,1\rbrace$ because of $\det_\R(1-T_z\varphi)\neq 0$. Just as in Corollary~\ref{quatlabel17} and Lemma~\ref{quatlabel1}, it follows that for $x\in M^\gamma$ there are exactly two fixed points of $\varphi$ contained in $\pi_Z^{-1}(x)\cap Z^\varphi$, namely $[p_x,[e_1]],[p_x,[e_2]]$ for a suitable choice of $p_x\in P_x$ as in Corollary~\hyperref[quatlabel17]{\ref*{quatlabel17}.(\ref*{quatlabel18})}.
\end{proof}
\begin{theorem}[Fixed point theorem for the Dolbeault complex of the twistor space]
\label{quatlabel23}
    Let $r\in 2\cdot\N_0$, $M$ be a compact quaternionic manifold and $\pi_F\colon F\rightarrow M$ a $B_2$-vector bundle. Furthermore, let $\varphi\colon Z\rightarrow Z$ be a holomorphic map and $\gamma\colon M\rightarrow M$ a map with $\gamma \circ \pi_Z = \pi_Z \circ \varphi$ such that $\det\nolimits_\R (1-T_z\varphi)\neq 0$ for all $z\in Z^\varphi$, $T_x\gamma\neq 0$ for all $x\in M^\gamma$ and, for $r\neq 0$, $\det\nolimits_\R (T_x\gamma)\neq 0$ for all $x\in M$. With the maps $\varphi^Z,\varphi^{\mathcal{L}^r},\varphi^{\pi_Z^\ast F}$ and $\tau^Z,\tau^{\mathcal{L}^r},\tau^{\pi_Z^\ast F}$ as in the previous section, such that $\tau^{\mathcal{L}^r\tenc\pi_Z^\ast F} := \tau^Z\tenc\tau^{\mathcal{L}^r}\tenc\tau^{\pi_Z^\ast F}$ commutes with $\dolb$, we obtain
    \begin{align*}
        \Tr_\mathrm{s} \left(\tau^{\mathcal{L}^r\tenc\pi_Z^\ast F}_{\vert H^{0,\bullet}_{\dolb}(Z,\mathcal{L}^r\tenc\pi_Z^\ast F)}\right) = \Tr_\mathrm{s} \left(\tau^{V^{\bullet,r}\tenc F}_{\vert H^\bullet_{\delta_{F,r}}(M)}\right) = \sum\limits_{x\in M^\gamma} \omega (x),
    \end{align*}
    where $T_x\gamma = p_x\circ (v\mapsto T_xve^{-it_x}) \circ p_x^{-1}$ for $x\in M^\gamma$ and $p_x\in P_x$ as in Corollary~\hyperref[quatlabel17]{\ref*{quatlabel17}.(\ref*{quatlabel18})} and
    \begin{align*}
        \omega (x) := \Tr (\gamma^F_{\vert x})\cdot \left(\frac{\det_\C (T_x)^{-\frac{r}{2(n+1)}}e^{irt_x}}{(1-e^{-2it_x})\det_\C (1-e^{-it_x}T_x)} + \frac{\det_\C (T_x)^{-\frac{r}{2(n+1)}}e^{-irt_x}}{(1-e^{2it_x})\det_\C (1-e^{it_x}T_x)}\right)
    \end{align*}
    for $e^{it_x}\not\in \lbrace -1,1\rbrace$.\\
    Here, the two summands of $\omega(x)$ are caused by the two fixed points in $Z_x$, thus the fixed point formula above is exactly the classic fixed point formula by Atiyah and Bott.
\end{theorem}
\begin{proof}
According to Lemma~\ref{quatlabel22}, the fixed points of $\varphi$ occur in pairs in $\pi_Z^{-1}(M^\gamma)$. For a fixed point $x\in M^\gamma$ with $T_x\gamma = p_x\circ (v\mapsto T_xve^{-it_x}) \circ p_x^{-1}$ for $x\in M^\gamma$ and $p_x\in P_x$ as in Corollary~\hyperref[quatlabel17]{\ref*{quatlabel17}.(\ref*{quatlabel18})}, $t_x\in \pi\cdot\Z$ can not be possible because otherwise we would have $\varphi_{\vert Z_x}=\Id_{Z_x}$ according to Corollary~\ref{quatlabel17}, which, however, contradicts the assumption of isolated fixed points. According to Theorem~\ref{quatlabel12}, $(\beta_{F,r}\circ\mu_{F,r})\colon H^\bullet_{\delta_{F,r}}(M)\rightarrow H^{0,\bullet}_{\dolb}(Z,\mathcal{L}^r\tenc\pi_Z^\ast F)$ is an isomorphism, and we have $\tau^{V^{\bullet,r}\tenc F}=(\beta_{F,r}\circ\mu_{F,r})^{-1} \circ \tau^{\mathcal{L}^r\tenc\pi_Z^\ast F} \circ (\beta_{F,r}\circ\mu_{F,r})$.\\
We choose a fixed torsion-free connection on the $G$-structure of $M$ and consider the induced decomposition $TZ=T^VZ\oplus T^HZ$. For $z\in Z^\varphi$ with $x:=\pi_Z(z)$, we have 
\begin{align*}
    T_z\varphi = \begin{pmatrix}
        T_z\varphi_{\vert Z_x} & \star\\
        0 & ({T_z\pi_Z}_{\vert T^H_z Z})^{-1}\circ T_x\gamma \circ {T_z\pi_Z}_{\vert T^H_z Z}
    \end{pmatrix}.
\end{align*}
Therefore,
\begin{align*}
    0\neq &\det\nolimits_\R(1-T_z\varphi)=\det\nolimits_\R\begin{pmatrix}
        1-T_z\varphi_{\vert Z_x} & \star\\
        0 & 1-({T_z\pi_Z}_{\vert T^H_z Z})^{-1}\circ T_x\gamma \circ {T_z\pi_Z}_{\vert T^H_z Z}
    \end{pmatrix}\\
    =& \det\nolimits_\R(1-T_z\varphi_{\vert Z_x})\det\nolimits_\R(1-({T_z\pi_Z}_{\vert T^H_z Z})^{-1}\circ T_x\gamma \circ {T_z\pi_Z}_{\vert T^H_z Z})\\
    =& \det\nolimits_\R(1-T_z\varphi_{\vert Z_x})\det\nolimits_\R(1-T_x\gamma),
\end{align*}
where $\varphi_{\vert Z_x}$ is the restriction of $\varphi$ to the fibre $Z_x$. We obtain $\det\nolimits_\R(1-T_x\gamma)\neq 0$.\\
The first statement thus follows from Theorem~\ref{quatlabel33}.\\
For the second statement, we first note that we can of course also compute this directly by applying Theorem~\ref{quatlabel14} to the Dolbeault complex. The general application to the Dolbeault complex can be found in \cite[Theorem~4.12]{atiyahbott2}, from which it follows that
\begin{align*}
    \omega(x):=\sum\limits_{z\in Z^\varphi\cap Z_x}\frac{\Tr \left(\varphi^{\mathcal{L}^r}_{\vert z}\tenc\varphi^{\pi_Z^\ast F}_{\vert z}\right)}{\det_\C (1-T^{1,0}_z\varphi)}
\end{align*}
holds, which we will compute now. With regard to the decomposition $TZ=T^VZ\oplus T^HZ$, we have $\mathcal{J}=\begin{pmatrix}
        \mathcal{J}_{\vert T^VZ} & 0\\
        0 & \mathcal{J}_{\vert T^H Z}
    \end{pmatrix}$ and therefore $T^{1,0}Z=T^{1,0,V}Z\oplus T^{1,0,H}Z$. We have $T^{1,0,H}_zZ\cong T^{1,0}M_{J_z}$ by applying $T_z\pi_Z\tenr \Id_\C$ due to the construction of $\mathcal{J}_{\vert T^H Z}$. In summary, we obtain
\begin{align*}
    T^{1,0}_z\varphi = \begin{pmatrix}
        {T_z\varphi_{\vert Z_x}\tenr \Id_\C}_{\vert T^{1,0,V}_zZ} & \star\\
        0 & {(({T_z\pi_Z}_{\vert T^H_z Z})^{-1}\circ T_x\gamma \circ {T_z\pi_Z}_{\vert T^H_z Z})\tenr \Id_\C}_{\vert T_z^{1,0,H}Z}
    \end{pmatrix}
\end{align*}
The two fixed points in $Z^\varphi\cap Z_x$ are $z_{x,1}:=[p_x,[e_1]],z_{x,
2}:=[p_x,[e_2]]$.\\
The map $\varphi_{\vert Z_x}$ is holomorphic and we can determine $ {T_z\varphi_{\vert Z_x}\tenr \Id_\C}_{\vert T^{1,0,V}_zZ}$ using the complex structure of the fibres. It is easy to see that ${T_{z_{x,1}}\varphi_{\vert Z_x}\tenr \Id_\C}_{\vert T^{1,0,V}_{z_{x,1}}Z}=e^{-2it_x}$ and ${T_{z_{x,2}}\varphi_{\vert Z_x}\tenr \Id_\C}_{\vert T^{1,0,V}_{z_{x,2}}Z}=e^{2it_x}$ because the restriction of the map to the fibres is exactly $[h]\mapsto[\begin{pmatrix}
    e^{it_x}&0\\0&e^{-it_x}
\end{pmatrix}h]$.\\
Similar to the proof of Theorem~\ref{quatlabel33}, together with (\ref{quatlabel7}) and $T^{1,0}_zZ\cong T^{1,0}M_{J_z}$ we obtain
\begin{align*}
    &\det\nolimits_\C (1-{(({T_z\pi_Z}_{\vert T^H_z Z})^{-1}\circ T_x\gamma \circ {T_z\pi_Z}_{\vert T^H_z Z})\tenr \Id_\C}_{\vert T_z^{1,0,H}Z})\\
    =& \det\nolimits_\C (1- {T_x\gamma\tenr \Id_\C}_{\vert T_x^{1,0}Z_{J_z}})\\
    =& \begin{cases}
        \det\nolimits_\C(1-e^{-it_x}T_x)\text{ for } z={z_{x,1}},\\
        \det\nolimits_\C(1-e^{it_x}T_x)\text{ for } z={z_{x,2}}.
    \end{cases}
\end{align*}
According to Lemma~\ref{quatlabel20}, we have
\begin{align*}
    \Tr \left(\varphi^{\mathcal{L}^r}_{\vert z}\tenc\varphi^{\pi_Z^\ast F}_{\vert z}\right) =& \Tr \left(\varphi^{\mathcal{L}^r}\right) \Tr\left(\varphi^{\pi_Z^\ast F}\right) = \begin{cases}
        \det\nolimits_\C (T_x)^{-\frac{r}{2(n+1)}}e^{rit_x}\Tr (\gamma^F_{\vert x}) \text{ for } z = z_{x,1},\\
        \det\nolimits_\C (T_x)^{-\frac{r}{2(n+1)}}e^{-rit_x}\Tr (\gamma^F_{\vert x}) \text{ for } z = z_{x,2}.
    \end{cases}
\end{align*}
In summary, we have
\begin{align*}
    \frac{\Tr \left(\varphi^{\mathcal{L}^r}_{\vert z}\tenc\varphi^{\pi_Z^\ast F}_{\vert z}\right)}{\det_\C (1-T^{1,0}_z\varphi)}=\begin{cases}
        \Tr (\gamma^F_{\vert x})\frac{\det_\C (T_x)^{-\frac{r}{2(n+1)}}e^{irt_x}}{(1-e^{-2it_x})\det_\C (1-e^{-it_x}T_x)} \text{ for } z=z_{x,1},\\
        \Tr (\gamma^F_{\vert x})\frac{\det_\C (T_x)^{-\frac{r}{2(n+1)}}e^{-irt_x}}{(1-e^{2it_x})\det_\C (1-e^{it_x}T_x)} \text{ for } z=z_{x,2}.
    \end{cases}
\end{align*}
\end{proof}
\subsection{Fixed Point Theorem for Wolf Spaces}
\label{quatsect4.4}
An interesting class of quaternionic manifolds are Wolf spaces, which are symmetric spaces that were classified in \cite{wolf}. Symmetric quaternionic manifolds always have constant scalar curvature, but the spaces with negative scalar curvature are not compact. Spaces with positive scalar curvature, on the other hand, are compact, and these are precisely the spaces $G/K$ for a compact simple Lie group without centre and its Lie subgroup $K$, which is the normaliser of the Lie subgroup generated by a maximal root. We will now consider an application of the fixed point theorems for exactly those spaces, since everything can be expressed with the help of roots. First, we give a brief summary of \cite{wolf} on the construction of these spaces. For the foundations of representation theory and the theory of compact Lie groups, the reader is referred to \cite{broetomdieck} and \cite{sepanski}.\\

Let $T=\frt/\Lambda\subset G$ be a maximal torus of a compact simple Lie group $G$ that has no centre, and let $\frg,\frt$ be the associated Lie algebras. The Killing form $B$ of $G$ provides a non-degenerate bilinear form on $\frg$, and therefore let $H_\alpha\in\frt$ for $\alpha\in\frt^\ast$ be defined by $\alpha=B(H_\alpha,\cdot)$.\\
Let $R\subset \Lambda^\ast:=\lbrace \lambda\in\frt^\ast\mid\lambda(\Lambda)\subset \Z\rbrace$ be the set of roots of $G$. With respect to a choice of positive roots $R^+\subset R$, let $\beta$ be the maximum root. Furthermore, let $R_{0,\beta}=\lbrace\alpha\in R\mid B(H_\alpha,H_\beta)=0\rbrace$ and $R_{0,\beta}^+:=R^+\cap R_{0,\beta}$. With the weight spaces $\frg_\alpha$ of a root $\alpha\in R$, we define 
\begin{align*}
    \frk' &:=\lbrace H\in \frt\mid \beta (H)=0\rbrace\oplus \bigoplus\nolimits_{\alpha \in R^+_{0,\beta}} \frg\cap(\frg_\alpha\oplus\frg_{-\alpha}),\\
    \fra &:=\frg\cap(\C H_\beta\oplus\frg_\beta\oplus \frg_{-\beta}),\\
    \frp &:= \bigoplus\nolimits_{\alpha \in R^+\setminus (\lbrace \beta\rbrace\cup R^+_{0,\beta})} \frg\cap(\frg_\alpha\oplus\frg_{-\alpha})
\end{align*}
and $K',A,K_1:=K'\cdot A$ to be the corresponding Lie subgroups of $G$ to the Lie algebras $\frk',\fra,\frk_1:=\frk'\oplus\fra$. In particular, $A\cong \spn{1}$ and $K'$ is isomorphic to a subgroup of $\spn{n}$ for some $n\in\N$, so $K_1\subset \spn{n}\cdot \spn{1}$. Then $M:=G/K_1$ is a Wolf space with the twistor space $Z:=G/K_2$ with $L:=\exp_G (\R\cdot H_\beta)\subset A, K_2:=K'\cdot L$. The Lie algebra of $K_2$ is $\frk_2:=\frk'\oplus \R \cdot H_\beta$.\\
We now have $T_{eK_2} G/K_2=\bigoplus\limits_{\alpha \in R^+\setminus R^+_{0,\beta}} \frg_\alpha\oplus\frg_{-\alpha}$ and the $G$-invariant complex structure of $G/K_2$ is chosen in such a way that
\begin{align*}
    T_{eK_2}^{1,0} G/K_2 = \bigoplus\limits_{\alpha \in R^+\setminus R^+_{0,\beta}} \frg_\alpha&&\text{and}&&
    T_{eK_2}^{0,1} G/K_2 = \bigoplus\limits_{\alpha \in R^+\setminus R^+_{0,\beta}} \frg_{-\alpha}.
\end{align*}
We denote by $L_g$ the action of $g\in G$ on $G/K_1$ and $G/K_2$, respectively. For every $g\in G$, the map $L_g\colon G/K_2\rightarrow G/K_2$, $g'K_2\mapsto gg'K_2$ is a holomorphic map, and $L_g\colon G/K_1\rightarrow G/K_1$ is therefore a quaternionic map. Since every element of $G$ is contained in a maximal torus and any two maximal tori are conjugate to each other, $g$ is conjugate to an element $t\in T$, thus it is sufficient to consider maps $L_t\colon G/K_2\rightarrow G/K_2$ for $t\in T$. In order to ensure that we obtain non-degenerate fixed points, we additionally require that the considered $t=\exp_G(X)\in T$ generates the torus, i.e.\ $\overline{\langle t\rangle_\Z}=T$ must hold.\\
Note that we have $T\subset K_2\subset K_1$. The fixed points of $L_t$ are always non-degenerate on $G/K_1$ and are given by $n_1K_1,\ldots,n_kK_1$ for $W_G/W_{K_1}=\lbrace (n_1T)W_{K_1},\ldots,(n_kT)W_{K_1}\rbrace$ and $k=\vert W_G/W_{K_1}\vert$.\\
The fixed points of $L_t$ on $G/K_2$ are also non-degenerate. With the map $\Theta \colon \frg\otimes_\R\C\rightarrow \frg\otimes_\R \C$, $Y\otimes z\mapsto Y\otimes \overline{z}$, let $X_j\in\frg_\beta\setminus \lbrace 0\rbrace$ be chosen such that $[X_j,-\Theta(X_j)]=\frac{H_\beta} {\pi i B(H_\beta,H_\beta)}$. Consider $a_j:=\exp_G(\frac{\pi}{2} J)\in A$ for $J:=-(X_j+\Theta (X_j))$. It follows that the fixed points of $L_t$ on $G/K_2$ correspond exactly to the elements of
\begin{align*}
    \lbrace (n_1T)W_{K_2},(n_1a_jT)W_{K_2}, \ldots, (n_kT)W_{K_2},(n_ka_jT)W_{K_2}\rbrace = W_G/W_{K_2}.
\end{align*}
Let $gK_2\in\lbrace n_1K_2,n_1a_jK_2, \ldots, n_kK_2,n_ka_jK_2\rbrace$ be one of the fixed points of $L_t\colon G/K_2\rightarrow G/K_2$. Then $gTg^{-1}\subset T$, $g^{-1}Tg\subset T$ and $T_{gK_2}L_t$ has the same eigenvalues as
\begin{align*}
    (T_{eK_2}L_g)^{-1}\circ T_{gK_2}L_t\circ T_{eK_2}L_g = T_{eK_2}L_{g^{-1}tg}\colon T_{eK_2}G/K_2 \rightarrow T_{eK_2}G/K_2.
\end{align*}
We have $T_{eK_2}L_{g^{-1}tg}={\Ad_{g^{-1}tg}}_{\vert \frg/\frk_2 }$ due to $g^{-1}tg=g^{-1}\exp_G(X)g=\exp_G(\Ad_{g^{-1}}(X))$ which also yields $\Ad_{g^{-1}tg} (v)=e^{2\pi i\alpha (\Ad_{g^{-1}}(X))}v$ for $v\in \frg_\alpha$, i.e.\ 
\begin{align*}
    \Ad_t(\Ad_g(v))=e^{2\pi i\alpha (\Ad_{g^{-1}}(X))}\Ad_g(v).
\end{align*}

The element $w=gT\in W_G$ of the Weyl group acts on the roots by $w\cdot \alpha=\alpha\circ\Ad_{g^{-1}}$. Thus,
\begin{align*}
    \Ad_{g^{-1}}\circ \Ad_t\circ \underbrace{\Ad_g(v)}_{\in \frg_{w\cdot \alpha}} = \Ad_{g^{-1}} (e^{2\pi i(w\cdot \alpha)(X)}\Ad_g (v))=e^{2\pi i(w\cdot \alpha)(X)}v
\end{align*}
now follows for $\alpha\in\frg_\alpha$. Hence, $T_{gK_2}L_t$ on $T_{gK_2}G/K_2$ has the eigenvalues $e^{\pm 2\pi i (w\cdot \alpha)(X)}$ for $\alpha\in R^{+}\setminus R^+_{0,\beta} $ and on the holomorphic tangent bundle the eigenvalues are just $e^{2\pi i (w\cdot \alpha)(X)}$ for $\alpha\in R^{+}\setminus R^+_{0,\beta}$, where the eigenvalue for the vertical tangent bundle is $e^{2\pi i(w\cdot \beta)(X)}$.\\
As in \cite{kkquattorsion}, we now consider a complex representation $\rho_\mathcal{W}\colon K_1\rightarrow \Aut_\C (\mathcal{W})$ of $K_1$ such that $A\subset K_1$ acts trivially. Then the associated vector bundle $G\times_{K_1}\mathcal{W}$ yields a $B_2$-vector bundle, and the pullback vector bundle on $G/K_2$ is precisely $G\times_{K_2}\mathcal{W}$.\\
The action of $G$ on $G\times_{K_2}\mathcal{W}$ yields a biholomorphic map $L_t^\mathcal{W}\colon G\times_{K_2}\mathcal{W} \rightarrow L_t^\ast (G\times_{K_2}\mathcal{W})$, $[g,v]\mapsto (gK_2,[tg,v])$, whose inverse is precisely a map $(L_t^\mathcal{W})^{-1}\colon L_t^\ast (G\times_{K_2}\mathcal{W}) \rightarrow G\times_{K_2}\mathcal{W}$ that fulfils the conditions from the fixed point theorem. At a fixed point $gK_2$ of $L_t$ on $G/K_2$, we have $g^{-1}tg=:k\in K_2$ and $[t\cdot g,v]=[gk,v]=[g,\rho_\mathcal{W} (k)v]=[g,\rho_\mathcal{W} (g^{-1}tg)v]$. Thus, we obtain $\Tr_\C ((L_t^\mathcal{W})^{-1})_{\vert gK_2}=\Tr (\rho_\mathcal{W}(g^{-1}tg)^{-1})=\Tr (\rho_\mathcal{W}(g^{-1}t^{-1}g))$, which is exactly the character $\chi_{\rho_\mathcal{W}}(g^{-1}t^{-1}g)$ of the representation $\rho_\mathcal{W}$ at the point $w^{-1}\cdot t^{-1}= g^{-1}t^{-1}g$ for $w=gT\in W_G$. For the fixed points $n_lK_2,n_la_jK_2\in \pi_Z^{-1}(n_lK_1)$, we note that $\rho_\mathcal{W}(n_l^{-1}t^{-1}n_l)=\rho_\mathcal{W}(a_j^{-1}n_l^{-1}t^{-1}n_la_j)$ because $A$ acts trivially.\\
In summary, we obtain the following application of the fixed point theorem.
\begin{corollary}
\label{quatlabel24}
    Let $M:=G/K_1$ be a Wolf space with positive scalar curvature with maximal torus $T$ and twistor space $Z:=G/K_2$. For $r\in 2\cdot\N_0$ and a representation $\rho_\mathcal{W}\colon K_1\rightarrow \Aut_\C (\mathcal{W})$, such that $A\subset K_1$ acts trivially, we obtain
    \begin{align*}
        &\Tr_\mathrm{s} \left(\tau^{\mathcal{L}^r\tenc (G\times_{K_2}\mathcal{W})}_{\vert H^{0,\bullet}_{\dolb}(G/K_2,\mathcal{L}^r\tenc (G\times_{K_2}\mathcal{W}))}\right) = \Tr_\mathrm{s} \left(\tau^{V^{\bullet,r}\tenc G\times_{K_1}\mathcal{W}}_{\vert H^\bullet_{\delta_{G\times_{K_1}\mathcal{W},r}}(G/K_1)}\right) \\
        =& \begin{aligned}[t]
            \sum\limits_{[w]\in W_G/W_{K_1}} &\chi_{\rho_\mathcal{W}} (w^{-1}\cdot t^{-1})\\
            &\cdot\left(\frac{e^{r\pi i (w\cdot \beta)(X)}}{\prod\limits_{\alpha \in R^+\setminus R^+_{0,\beta}} (1-e^{-2\pi i (w\cdot \alpha)(X)})} + \frac{e^{-r\pi i (w\cdot \beta)(X)}}{\prod\limits_{\alpha \in R^+\setminus R^+_{0,\beta}} (1-e^{2\pi i (w\cdot \alpha)(X)})} \right)
        \end{aligned}\\
        =& \sum\limits_{[w]\in W_G/W_{K_2}}  \frac{\chi_{\rho_\mathcal{W}} (w^{-1}\cdot t^{-1}) e^{r\pi i (w\cdot \beta)(X)}}{\prod\limits_{\alpha \in R^+\setminus R^+_{0,\beta}} (1-e^{-2\pi i (w\cdot \alpha)(X)})}
    \end{align*}
    for a generating element of the torus $t:=\exp_G(X)\in T$.
\end{corollary}
\hfill\break
The positive roots of $K_2$ are $R^+_{0,\beta}$. Therefore, we define $\rho_G:=\frac{1}{2}\sum\limits_{\alpha\in R^+} \alpha$ and $\rho_{K_2}:=\frac{1}{2}\sum\limits_{\alpha\in R^+_{0,\beta}} \alpha$ to be half the sum of all positive roots.\\
For $\lambda\in\Lambda^\ast$, set $\operatorname{Alt}_G(\lambda):=\sum\limits_{w\in W_G} (-1)^w e^{2\pi i(w\cdot \lambda)}$ where $(-1)^w$ denotes the sign of $w\in W_G$. The (formal) character is defined as
\begin{align*}
    \chi_{\rho_G+\lambda}(t):=\frac{\operatorname{Alt}_G(\rho_G+\lambda)(t)}{\operatorname{Alt}_G(\rho_G)(t)}.
\end{align*}
Let $\mathcal{W}:=\bigoplus\limits_{\lambda \in \Delta (\mathcal{W})}\mathcal{W}_\lambda$ be the decomposition of $\mathcal{W}$ into weight spaces, where $\Delta (\mathcal{W})\subset \Lambda^\ast$ is the set of weights and $m_\lambda:=\dim_\C (\mathcal{W}_\lambda)$.
\begin{lemma}
    In the setting of Corollary~\ref{quatlabel24} with $r=2r'$, we have
    \begin{align*}
        \sum\limits_{[w]\in W_G/W_{K_2}}  \frac{\chi_{\rho_\mathcal{W}} (w^{-1}\cdot t^{-1}) e^{r\pi i (w\cdot \beta)(X)}}{\prod\limits_{\alpha \in R^+\setminus R^+_{0,\beta}} (1-e^{-2\pi i (w\cdot \alpha)(X)})} = \sum\limits_{\lambda\in\Delta(\mathcal{W})} m_\lambda \chi_{\rho_G+r'\beta-\lambda}(t).
    \end{align*}
\end{lemma}
\begin{proof}
    Since all positive roots $R^+_{0,\beta}$ of $K_2$ are orthogonal to $\beta$, and the Weyl group $W_{K_2}$ is generated by reflections in these roots, it follows that $\beta$ is fixed by every $w'\in W_{K_2}$.\\
    In particular, as $\mathcal{W}$ is a representation of $K_2$, we have $\chi_{\rho_\mathcal{W}}(w'\cdot (w^{-1}\cdot t^{-1}))=\chi_{\rho_\mathcal{W}}(w^{-1}\cdot t^{-1})=\chi_{\rho_\mathcal{W}}({w'}^{-1}\cdot (w^{-1}\cdot t^{-1}))$ because $w'=gT\in W_{K_2}$ acts by conjugation on $T$. Due to the decomposition into weight spaces, we obtain
    \begin{align*}
        \chi_{\rho_\mathcal{W}}({w'}^{-1}\cdot (w^{-1}\cdot t^{-1}))=\sum\limits_{\lambda\in \Delta(\mathcal{W})} m_\lambda e^{2\pi i \lambda ({w'}^{-1}\cdot w^{-1}\cdot (-X))} = \sum\limits_{\lambda\in \Delta(\mathcal{W})} m_\lambda e^{-2\pi i (w\cdot w'\cdot \lambda) (X)}.
    \end{align*}
    With a calculation similar to \cite[Equation~(76)]{kkhermsymm} and due to 
    \begin{align*}
        \sum\limits_{w\in W_G} (-1)^w e^{2\pi i (w\cdot \rho_G) (X)} = \prod\limits_{\alpha\in R^+}2i\sin (\pi\alpha(X)),
    \end{align*}
    we have
    \begin{align*}
        &\sum\limits_{[w]\in W_G/W_{K_2}}  \frac{\chi_{\rho_\mathcal{W}} (w^{-1}\cdot t^{-1}) e^{r\pi i (w\cdot \beta)(X)}}{\prod\limits_{\alpha \in R^+\setminus R^+_{0,\beta}} (1-e^{-2\pi i (w\cdot \alpha)(X)})}\\
        =& \sum\limits_{[w]\in W_G/W_{K_2}} \frac{e^{2\pi i (\rho_G - \rho_{K_2})(X)}\prod\limits_{\alpha\in R^+_{0,\beta}} 2i\sin (\pi(w\cdot \alpha)(X))}{\prod\limits_{\alpha\in R^+}2i\sin (\pi(w\cdot \alpha)(X))}\chi_{\rho_\mathcal{W}} (w^{-1}\cdot t^{-1}) e^{2\pi i r' (w\cdot \beta)(X)}\\
        =& \begin{aligned}[t]
            \sum\limits_{[w]\in W_G/W_{K_2}}& \frac{e^{2\pi i (\rho_G - \rho_{K_2})(X)} \sum\limits_{w'\in W_{K_2}} (-1)^{w'} e^{2\pi i (w\cdot w'\cdot \rho_{K_2})(X)} }{(-1)^{w}\prod\limits_{\alpha\in R^+}2i\sin (\pi\alpha(X))}\\
            &\cdot\chi_{\rho_\mathcal{W}} ({w'}^{-1}\cdot (w^{-1}\cdot t^{-1})) e^{2\pi i r' (w\cdot w'\cdot \beta)(X)}
        \end{aligned}\\
        =& \sum\limits_{\lambda\in \Delta(\mathcal{W})} m_\lambda \sum\limits_{[w]\in W_G/W_{K_2}} \frac{\sum\limits_{w'\in W_{K_2}} (-1)^w(-1)^{w'} e^{2\pi i (w\cdot w'\cdot (\rho_G + r'\beta -\lambda))(X)}}{\prod\limits_{\alpha\in R^+}2i\sin (\pi\alpha(X))}\\
        =& \sum\limits_{\lambda\in \Delta(\mathcal{W})} m_\lambda \sum\limits_{w\in W_G} \frac{(-1)^w e^{2\pi i (w\cdot (\rho_G + r'\beta -\lambda))(X)}}{\prod\limits_{\alpha\in R^+}2i\sin (\pi\alpha(X))} \\
        =& \sum\limits_{\lambda\in \Delta(\mathcal{W})} m_\lambda \frac{\sum\limits_{w\in W_G}(-1)^w e^{2\pi i (w\cdot (\rho_G + r'\beta -\lambda))(X)}}{\sum\limits_{w\in W_G} (-1)^w e^{2\pi i (w\cdot \rho_G) (X)}}
        = \sum\limits_{\lambda\in\Delta(\mathcal{W})} m_\lambda \chi_{\rho_G+r'\beta-\lambda}(t).
    \end{align*}
\end{proof}
\begin{example}
    We will now take a closer look at the Wolf space $M := \SU (n+2)/\operatorname{S}(\operatorname{U}(n)\times \operatorname{U}(2))$, whose twistor space is $Z := \SU (n+2)/\operatorname{S}(\operatorname{U}(n)\times \operatorname{U}(1)\times \operatorname{U} (1))$. Therefore, $G= \SU (n+2)$, $K_1=\operatorname{S}(\operatorname{U}(n)\times \operatorname{U}(2))$ and $K_2=\operatorname{S}(\operatorname{U}(n)\times \operatorname{U}(1)\times \operatorname{U} (1))$.
The centre $\SU (n+2)$ is non-trivial, but as the centre is a finite subgroup resulting from  unit roots and is also a subgroup of $\operatorname{S}(\operatorname{U}(n)\times \operatorname{U}(2))$, the quotient groups still fall under the classification by \cite{wolf}.\\
With $\operatorname{diag}(a_1,\ldots,a_{n+2})\in \C^{(n+2)\times (n+2)} $, we denote the diagonal matrix with the entries $a_1,\ldots,a_{n+2}$.\\
In \cite[Section~V.6]{broetomdieck} and \cite[Chapter~6]{sepanski}, the following details about $\SU (n+2)$ can be found. We consider the maximal torus $T:= \lbrace \operatorname{diag} (e^{i t_1},\ldots,e^{i t_{n+2}})\mid t_1+\ldots+ t_{n+2} = 0 \mod 2\pi \rbrace$. Hence, $W_{\SU (n+2)} = S_{n+2}$ is the symmetric group of permutations and $W_{\operatorname{S}(\operatorname{U}(n)\times \operatorname{U}(2))} = S_n\times S_2$, where $S_n$ permutes the elements of $\lbrace 1,\ldots,n\rbrace$ and $S_2$ permutes the elements of $\lbrace n+1, n+2\rbrace$. There is a bijection between $S_{n+2}/(S_n\times S_2)$ and $\lbrace \lbrace k,l\rbrace\mid 1\leq k<l\leq n+2\rbrace$ given by $\sigma(S_n\times S_2)\mapsto \lbrace \sigma(n+1),\sigma(n+2)\rbrace$.\\
The roots are $R=\lbrace e_k-e_l\mid 1\leq k,l\leq n+2 ,k\neq l\rbrace$. With the positive roots being $R^+=\lbrace e_k-e_l\mid 1\leq k<l\leq n+2\rbrace$, the maximum root is $\beta:= e_{n+1}-e_{n+2}$. We write $\alpha_{k,l}:=e_k-e_l$ and the Weyl group acts by permuting the indices. Therefore, $R^+_{0,\beta}=\lbrace \alpha_{k,n+1}\mid 1\leq k<n+1\rbrace\cup \lbrace \alpha_{k,n+2}\mid 1\leq k<n+2\rbrace$.\\
Let $X:=\operatorname{diag}(it_1,\ldots,it_{n+2})\in\frt$, such that $t:=\exp_{\SU (n+2)}(X)$ generates the maximal torus $T$. According to \cite{Kronecker}, this is exactly the case if $1,t_1,\ldots,t_{n+2}$ are linearly independent over $\mathbb{Q}$.\\
For a representation $\rho_\mathcal{W}\colon \operatorname{S}(\operatorname{U}(n)\times \operatorname{U}(2))\rightarrow \Aut_\C(\mathcal{W})$ with ${\rho_\mathcal{W}}_{\vert \SU (2)}\equiv \Id_{\mathcal{W}}$, we obtain
    \begin{align*}
        &\Tr_\mathrm{s} \left(\tau^{\mathcal{L}^r\tenc (G\times_{K_2}\mathcal{W})}_{\vert H^{0,\bullet}_{\dolb}(G/K_2,\mathcal{L}^r\tenc (G\times_{K_2}\mathcal{W}))}\right) = \Tr_\mathrm{s} \left(\tau^{V^{\bullet,r}\tenc G\times_{K_1}\mathcal{W}}_{\vert H^\bullet_{\delta_{G\times_{K_1}\mathcal{W},r}}(G/K_1)}\right)\\
        =& \begin{aligned}[t]
            \sum\limits_{1\leq k<l\leq n+2} \chi_{\rho_\mathcal{W}} (\sigma_{k,l}^{-1}\cdot t^{-1})\Bigg(& \frac{e^{r\pi i (t_k-t_l)}}{\left(1-e^{-2\pi i (t_k-t_l)}\right)\prod\limits_{\substack{1\leq k'< l'\leq n+2\\ k',l'\not\in \lbrace k,l\rbrace}} \left(1-e^{-2\pi i (t_{k'}-t_{l'})}\right)} \\
            & + \frac{e^{-r\pi i (t_k-t_l)}}{\left(1-e^{2\pi i (t_k-t_l)}\right)\prod\limits_{\substack{1\leq k'< l'\leq n+2\\ k',l'\not\in \lbrace k,l\rbrace}} \left(1-e^{2\pi i (t_{k'}-t_{l'})}\right)}\Bigg),
        \end{aligned}
    \end{align*}
where $\sigma_{k,l}\in S_{n+2}$ denotes a permutation with $\sigma_{k,l}(n+1)=k$ and $\sigma_{k,l}(n+2)=l$.
\end{example}
\printbibliography[title={References}]
\end{document}